\documentclass{amsart}
\usepackage{geometry}                
\usepackage{graphicx}
\usepackage{tikz}
\usetikzlibrary{shapes.geometric}
\usetikzlibrary{3d,calc}
\usepackage{amssymb}
\usepackage{epstopdf,amsmath}
\usepackage{amsthm}
\usepackage{xcolor}

\newtheorem{prop}{Proposition}[section]
\newtheorem{lemma}[prop]{Lemma}
\newtheorem{cor}[prop]{Corollary}
\newtheorem{theorem}[prop]{Theorem}

\newtheorem*{lemma*}{Lemma}
\newtheorem*{theorem*}{Theorem}

\newcommand{\eps}{\epsilon}

\newcommand{\RR}{\mathbb{R}}

\newcommand{\pp}{\mathbin{\!/\mkern-5mu/\!}}

\newcommand{\CP}{\mathbf{CP}}
\newcommand{\SSS}{\mathbf{S}}

\title{A sharp square function estimate for the cone in $\RR^3$}

\author{Larry Guth, Hong Wang, and Ruixiang Zhang}

\begin{document}

\begin {abstract} We prove a sharp square function estimate for the cone in $\mathbb{R}^3$ and {\color{black} consequently} the local smoothing conjecture for the wave equation in $2+1$ dimensions.
\end {abstract}

\maketitle

\section{Introduction}

\subsection{Main Results}

This paper concerns the restriction theory of the cone in $\RR^3$.  Let $\Gamma$ be the truncated light cone $\Gamma= \{\xi_1^2+\xi_2^2= \xi_3^2, 1/2\leq \xi_3\leq 1\}$, and let $N_{R^{-1}}(\Gamma)$ denote its $R^{-1}$-neighborhood.  Cover $N_{R^{-1}}(\Gamma)$ by {\color{black} finitely overlapping} sectors $\theta$ of angular width $R^{-1/2}$, where each sector is a rectangular box of dimensions about $R^{-1} \times R^{-1/2} \times 1$.  {\color{black}If $ \hat f$ has support on $N_{R^{-1}} (\Gamma)$, we consider a set of functions $\{f_{\theta}\}$ such that: (a) $\hat f_\theta$ is supported on $\theta$ and (b) $f = \sum_{\theta} f_{\theta}$. For example{\color{black}\footnote{\color{black}We remark that what we prove about $\{f_{\theta}\}$ in this paper is uniform as long as (a) and (b) are satisfied, i.e. does not depend on the particular choice of $\{f_{\theta}\}$. 
}} here is a natural way to choose $\{f_{\theta}\}$: let} $\psi_\theta$ be a smooth partition of unity {\color{black}subordinate} to {\color{black}the} covering {\color{black}$\{\theta\}$}, and define $f_\theta$ by $\hat f_\theta = \hat f \psi_\theta$.  We prove the following sharp square function estimate for this decomposition.

\begin{theorem}\label{sq fun} (Square function estimate)
	For any $\eps > 0$, $R \ge 1$ {\color{black}and} any function $f$ whose Fourier transform is supported on  $N_{R^{-1}}(\Gamma)$, we have
	$$\|f\|_{L^4({\color{black} \mathbb{R}^3})} \leq C_{\epsilon} R^{\epsilon} \left\|(\sum_\theta |f_{\theta}|^2)^{1/2} \right\|_{L^4{\color{black} (\mathbb{R}^3)}}.$$
\end{theorem}

This type of square function estimate was considered by Mockenhaupt \cite{M} who proved that it implies the cone multiplier conjecture in $\RR^3$, and by Mockenhaupt--Seeger--Sogge \cite{MSS} {\color{black} (in a slightly different form)} who {\color{black} essentially} showed that it implies the local smoothing conjecture for the wave equation in 2+1 dimensions.  Here we recall the local smoothing conjecture, and we refer to \cite{M} and \cite{LV} for more information about the cone multiplier conjecture. The local smoothing conjecture was formulated by Sogge in \cite{S}. If $u$ is a solution to the wave equation on $\RR^n$, a local smoothing inequality bounds $\| u \|_{L^p(\RR^n \times [1, 2])}$ in terms of the Sobolev norms of the initial data.  In particular, the local smoothing conjecture in 2 + 1 dimensions is the following estimate.

\begin{theorem} \label{locsmooth} (Local smoothing in 2+1 dimensions) Suppose that $u(x,t)$ is a solution of the wave equation in $2+1$ dimensions, with initial data $u(x,0) = u_0(x)$ and $\partial_tu(x,0) = u_1(x)$.  Then for any $p \ge 4$, and any $\alpha > \frac{1}{2} - \frac{2}{p}$,
	
\begin{equation} \label{locsmootheqn} \| u \|_{L^p (\RR^2 \times [1,2])} \le C_\alpha \left( \| u_0 \|_{p, \alpha} + \| u_1 \|_{p, -1 + \alpha} \right).\end{equation}

\end{theorem}

\noindent Theorem \ref{locsmooth} follows by combining Theorem \ref{sq fun} with {\color{black} the arguments in} \cite{MSS}.

In \cite{S}, Sogge formulated the local smoothing conjecture, and he noticed that Bourgain's proof of the boundedness of the circular maximal operator in \cite{B} can be used to establish ``local smoothing'' estimates with a nontrivial gain of regularity.  The critical case of Theorem \ref{locsmooth} is when $p=4$ and $\alpha$ is close to zero.  Mockenhaupt, Seeger, and Sogge \cite{MSS} proved that (\ref{locsmootheqn}) holds for $p=4$ with $\alpha > 1/8$, and this was improved afterwards by several authors (\cite{TV}, \cite{W2}{\color{black}, \cite{L}}).
In \cite{W}, Wolff proved the local smoothing conjecture for $p \geq 74$ in the full range\footnote{To be more specific, Sogge originally made the conjecture for $\alpha$ in the range $\alpha > \frac{1}{2}-\frac{2}{p}$ and Wolff confirmed Sogge's conjecture for $p\geq 74$ and $\alpha$ in this range. Later in the work \cite{HNS} of Heo, Nazarov and Seeger it was conjectured further that when $p>4$ the conjecture should hold for $\alpha \geq \frac{1}{2}-\frac{2}{p}$.} of $\alpha$. In that paper, Wolff introduced the idea of decoupling.  His method was extended to higher dimensions by {\L}aba--Wolff \cite{LW} and refined by {\color{black}Garrig{\'o}s--Seeger \cite{GS}\cite{GS2} and Garrig{\'o}s--Schlag--Seeger \cite{GSS}.}  Then in \cite{BD}, Bourgain and Demeter proved a sharp decoupling estimate for the cone in every dimension, in particular proving the local smoothing conjecture {\color{black}in $2+1$ dimensions} for $p \ge 6$ in the full range of $\alpha$.  The sharp decoupling estimate for the cone does not, however, imply the full range of local smoothing estimates -- at the end of the introduction we will discuss what the issue is.

In a different direction, Lee and Vargas \cite{LV} proved a sharp $L^3$ square function estimate using multilinear restriction.

\subsection{Proof Strategy}

One new feature of our approach is that we prove a stronger estimate which works better for induction on scales.  We need a little notation to state this estimate. The precise details and definitions are provided in Section 3.  First we recall the {\color{black} locally constant property} of $f$.  For each sector $\theta$, we let $\theta^*$ denote the dual rectangular box: since $\theta$ has dimensions $1 \times R^{-1/2} \times R^{-1}$, $\theta^*$ has dimensions $1 \times R^{1/2} \times R$.  We call such a $\theta^*$ a plank.  Recall that $|f_\theta|$ is roughly constant{\color{black}\footnote{\color{black} Such kind of ``locally constant'' heuristic will be used a few times in the current paper. To  justify this intuition one can use Corollary 4.3 in \cite{BD2}.    See also Lemma~\ref{lem: locally constant} and Lemma~\ref{lem: convolution} in Section \ref{secparab} of the current paper.}} on each translated copy of $\theta^*$. {\color{black} In this paper we tile $\mathbb{R}^3$ with translated copies of $\theta^*$.} The restriction of $f_\theta$ to one translated copy of $\theta^*$ is called a {\color{black} \emph{wave packet}}.  In addition to the sectors $\theta$, we will consider larger angular sectors $\tau$ with any angle between $R^{-1/2}$ and 1.  We write $d(\tau)$ to denote this angle, which we call the {\color{black}\emph{aperture}} of $\tau$. 
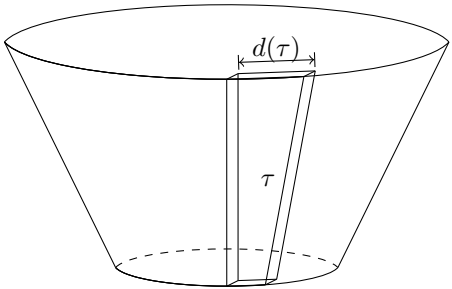
{

\begin{center}
	\begin{tikzpicture}[scale=1.5]
	\pgfmathsetmacro{\thickx}{0.1}
	\pgfmathsetmacro{\thicky}{0.05}
	\draw (0,0) arc(170:10:2cm and 0.4 cm) coordinate[pos=0] (a);
	\draw (0,0) arc(-170:-10:2cm and 0.4cm) coordinate (b);
	\draw (0,0) arc(-170: -90: 2cm and 0.4cm) coordinate(c);
	\draw (0,0) arc(-170: -70: 2cm and 0.4cm) coordinate(d);
	\coordinate(mid) at ([yshift=-4cm]$(a)!0.5!(b)$);
	\coordinate(start) at ($(mid)!0.5!(a)$);
	\draw (start)--(a);
	\draw[dashed] (start) arc(170:10:1cm and 0.2cm) coordinate[pos=0] (aa);
	\draw (start) arc(-170: -10: 1cm and 0.2cm) coordinate (bb);
	\draw (start) arc(-170: -90: 1cm and 0.2cm) coordinate (cc);
	\draw(start) arc(-170: -70: 1cm and 0.2cm) coordinate (dd);
	\draw (bb)--(b);
	\coordinate (t) at (\thickx, \thicky);
	\coordinate (c1) at ($(c)+ (t)$);
	\coordinate (cc1) at ($(cc)+(t)$);
	\coordinate (d1) at ($(d)+(t)$);
	\coordinate (dd1) at ($(dd)+(t)$);
	\draw[black] (c)--(c1);
	\draw[black] (c1)--(cc1);
	\draw[black] (cc)--(cc1);
	\draw[black](c)--(cc);
	\draw[black](d)--(dd);
	\draw[black](d)--(d1);
	\draw[black](dd)--(dd1);
	\draw[black] (d1)--(dd1);
	\draw[black] (c)--(d);
	\draw[black](cc)--(dd);
	\draw[black](c1)--(d1);
	\draw[black] (cc1)--(dd1);
	\node[left, black] at ($(d)!0.5!(dd)$) {$\tau$};
	\draw[|<->|, black] ($(c1)+( 0, 0.1)$)--($(d1)+(0, 0.1)$);
	\node[above, black] at ($(c1)!0.5!(d1)$) {$ d(\tau)$};

	\end{tikzpicture}
\end{center}

For each $\tau$, we define\footnote{This definition works best if $\tau$ is honestly tiled by $\theta$. In general we abuse the notation a bit: Throughout this paper, by writing ``summing over $\theta \subset \tau$'', we really mean ``summing over all $\theta \in A(\tau)$'' where the collection $A(\tau)$ is chosen as follows: Each $A(\tau)$ only contains those $\theta$'s who intersect $\tau$, and all $A(\tau)$ form a disjoint union $\{\theta\} = \bigsqcup_{\tau} A(\tau)$.}
$f_\tau = \sum_{\theta \subset \tau} f_\theta$, and we define $\tau^*$ to be the dual rectangle to $\tau$.  If $d(\tau) = s$, then $\tau^*$ has dimensions $1 \times s^{-1} \times s^{-2}$, and $|f_\tau|$ is roughly constant on each translated copy of $\tau^*$.  Next we define $U_{\tau,R}$ to be a scaled copy of $\tau^*$ with diameter $R$.  If $d(\tau) = s$, then $U_{\tau,R}$ has dimensions $R s^2 \times R s \times R$.  Note that if $\theta \subset \tau$ and if $T$ is a translated copy of $\theta^*$ which passes through the center of $U_{\tau,R}$, then $T \subset {\color{black}10}U_{\tau,R}$, where $10 U_{\tau,R}$ means the dilation of $U_{\tau,R}$ by a factor of $10$ with respect to its centroid.  For each $\tau$, we tile {\color{black} $\mathbb{R}^3$} by translated copies of $U_{\tau,R}$.

 $${\color{black} \mathbb{R}^3} = \bigsqcup_{U \textrm{ a translated copy of } U_{\tau,R}} U. $$

 \noindent This tiling is natural because for each $\theta \subset \tau$, the support of each wave packet of $f_\theta$ is essentially contained in $\sim 1$ tiles $U$ in the tiling.  Here two quantities  $A\sim B$  means that $A\leq C_1B\leq C_2A$ for some positive absolute constants $C_1$ and $C_2$.  We write $\sum_{U \pp U_{\tau,R}}$ to denote the sum over all the translated copies $U$ of $U_{\tau,R}$ in the tiling of {\color{black} $\mathbb{R}^3$}.

 {\color{black}If $U$ is a translated copy of $U_{\tau,R}$, then we define the square function $S_U f$ associated with $U$ to be
 $$S_U f = (\sum_{\theta \subset \tau}|f_{\theta}|^2)^{1/2} |_U.$$}



 We can now state our main estimate.
%

   \begin{theorem} \label{main} Suppose that $f$ has Fourier support on $N_{R^{-1}}(\Gamma)$.  Then
 	
\begin{equation} \label{L4strong}	\| f \|_{L^4({\color{black} \mathbb{R}^3})}^4  \le C_\eps R^\eps \sum_{R^{-1/2} \le s \le 1} \sum_{d(\tau) = s} \sum_{U \pp U_{\tau,R}} |U|^{-1} \| {\color{black}S_U f} \|_{L^2}^4 .\end{equation}
Here the sum over $s$ is over dyadic values of $s$ in the range $R^{-1/2} \le s \le 1$.
  \end{theorem}

Let us take a moment to digest the right-hand side of this estimate. {\color{black}For this discussion, suppose that $f$ is essentially supported on one $B_R$.} We start with the term where $s = R^{-1/2}$.  In this case $\tau$ is one of the original sectors $\theta$ of aperture $R^{-1/2}$,  $U_{\tau,R}$ is equal to $\theta^*$, and ${\color{black}|S_U f| = |f_\theta| \big|_U}$.  Since $|S_U f| = |f_\theta|$ is roughly constant on $U$, $$|U|^{-1} \| {\color{black}S_U f} \|_{L^2}^4 \sim \| {\color{black}S_U f} \|_{L^4}^4.$$ If the functions $f_\theta$ are essentially supported on disjoint regions, we would have
$$\| f \|_{L^4}^4 \sim \sum_{d(\theta) = R^{-1/2}} \sum_{U \pp U_{\tau,R}} \| {\color{black}S_U f} \|_{L^4}^4,$$ which matches the term $s = R^{-1/2}$ on the right-hand side of (\ref{L4strong}).  Next consider the term where $s=1$.  In this case, there is only one $\tau$ which covers all of $\Gamma$, and the contribution to the right-hand side is {\color{black}essentially} $|B_R|^{-1} \| S_{B_R} f \|_{L^2}^4 \sim |B_R|^{-1} \| f \|_{L^2(B_R)}^4$.   If $|f|$ is roughly constant on the whole $B_R$, then we would have $${\color{black}\| f \|_{L^4(\mathbb{R}^3)}^4 \sim } \| f \|_{L^4(B_R)}^4 \sim |B_R|^{-1} \| f \|_{L^2(B_R)}^4 \sim |B_R|^{-1} \| {\color{black}S_{B_R} f} \|_{L^2(B_R)}^4,$$ which matches the term $s = 1$ on the right-hand side of (\ref{L4strong}).  Finally we consider the intermediate values of $s$.  It may happen that $f = f_\tau$ for some $\tau$ and that $f$ is essentially supported on a particular translated copy $U$ of $U_{\tau,R}$ and that $|f|$ is roughly constant on $U$.  In this case, $${\color{black}\| f \|_{L^4(\mathbb{R}^3)}^4 \sim} \| f_\tau \|_{L^4({\color{black}U})}^4 \sim  |U|^{-1} \| f_\tau \|_{L^2(U)}^4 \sim |U|^{-1} \| {\color{black}S_U f} \|_{L^2}^4,$$ which is the term corresponding to $U$ on the right-hand side of (\ref{L4strong}).

The proof of Theorem \ref{main} is based on a new Kakeya-type estimate, which controls the overlapping of the planks in the wave packet decomposition of $f$.

\begin{lemma}\label{incidenceintro}  Suppose that $\hat f$ has support on $N_{R^{-1}}(\Gamma)$.  Let $g$ denote the {\color{black}(squared)} square function $g=\sum_{d(\theta)=R^{-1/2}} |f_{\theta}|^2$.  Then
	$$\int_{\color{black} \mathbb{R}^3} |g|^2 \lesssim \sum_{R^{-1/2} \leq s\leq 1}   \sum_{d(\tau)=s} \sum_{U\pp U_{\tau,R}} {\color{black} |U|^{-1}} \|{\color{black}S_U f}\|_{L^2}^4,$$
	where $A\lesssim B$ means that $A\leq C B$ for some absolute positive constant $C$.
\end{lemma}

\noindent Recall that each function $|f_\theta|$ is morally constant on the translated copies of $\theta^*$, where each $\theta^*$ is a $1 \times R^{1/2} \times R$ plank.  The estimate in Lemma \ref{incidenceintro} is a Kakeya-type bound on the overlapping of these planks.  The new feature of this estimate compared to previous Kakeya-type estimates is the structure of the right-hand side, which is designed to match the right-hand side of Theorem \ref{main}.  The terms on the right-hand side keep track of how planks are packed into the rectangular boxes $U$.  If the planks are spread out in the sense that each box $U$ does not contain too many planks, then it gives a strong bound.

In \cite{W}, Wolff connected Kakeya-type estimates for overlapping planks to incidence geometry problems in the spirit of the Szemer\'edi--Trotter problem.  He adapted the cutting method from incidence geometry to this setting and he used it to estimate the overlaps of planks.  He applied those geometric estimates at many scales to prove his results on local smoothing.  In \cite{BD}, Bourgain and Demeter apply multilinear Kakeya estimates at many scales to prove decoupling.  In this paper, we apply Lemma \ref{incidenceintro} at many scales to prove Theorem \ref{main}.

Lemma \ref{incidenceintro} is proven using Fourier analysis.  By Plancherel, $\int |g|^2 = \int |\hat g|^2$.  Roughly speaking, we decompose {\color{black}the} Fourier space, and the contributions of different regions to $\int |\hat g|^2$ correspond to the different terms on the right-hand side of Lemma \ref{incidenceintro}.  This approach to proving Kakeya-type estimates is based on some work of Orponen in projection theory \cite{O} and {\color{black}is related} to Vinh's work \cite{V} about incidence geometry over finite fields.  It builds on \cite{GSW}, which applies similar ideas to rectangles and tubes instead of planks.

\subsection{Local estimates}

{\color{black}Our Theorem \ref{main} and Lemma \ref{incidenceintro} have ``local'' counterparts involving polynomially decaying weights that are essentially supported on a given box. For any box $B_R$ of diameter $R$, define the weight $$w_{B_R, E} (x) = (1+\frac{\mathrm{dist} (x, B_R)}{R})^{-E}.$$

Here is the local version of Theorem \ref{main}.

\begin{theorem} \label{mainlocal} If $f$ has Fourier support on $N_{R^{-1}}(\Gamma)$, then for any $E > 0$,
 	
\begin{equation} \label{L4strong'}	\| f \|_{L^4(B_R)}^4  \le C_{\eps, E} R^\eps \sum_{R^{-1/2} \le s \le 1} \sum_{d(\tau) = s} \sum_{U \pp U_{\tau,R}} |U|^{-1} \| {w_{B_R, E} \cdot \color{black}S_U f}\|_{L^2}^4 .\end{equation}
Here the sum over $s$ is over dyadic values of $s$ in the range $R^{-1/2} \le s \le 1$.
\end{theorem}

In the above theorem, the sum on the right-hand side is also ``morally localized''. It is $$\sum_{R^{-1/2} \le s \le 1} \sum_{d(\tau) = s} \sum_{U \pp U_{\tau,R}, U \subset 100B_R} |U|^{-1} \| {\color{black}S_U f} \|_{L^2}^4$$
plus some decaying error term.  To prove Theorem \ref{mainlocal}, we multiply $f$ by  a rapidly decaying bump function $\phi_R$ adapted to $B_R$ such that $|\phi_R|> \frac{1}{C} >0$ on $B_R$ and $\hat \phi_R$ is supported on the ball $B_{R^{-1}}$ centered at the origin, and then we apply Theorem \ref{main} to the decomposition $\phi_R f = \sum_{\theta} \phi_R f_{\theta}$.

}}

\subsection{Relationship with decoupling}

While working on this project, we were strongly influenced by ideas related to decoupling, but the proof given here does not use the decoupling theorem per se.  It does make use of a nice observation that Bourgain and Demeter used to reduce the decoupling theorem for the cone to the decoupling theorem for the paraboloid (See \cite{BD}{\color{black}. Similar ideas can also be traced back to the iteration argument of Pramanik--Seeger \cite{PS}}).  Instead of working with a truncated cone of height 1, {\color{black}Bourgain and Demeter} worked with a truncated cone of height $1/K$ for a large constant $K$, denoted $\Gamma_{\frac{1}{K}}$.  This shorter truncated cone can be approximated by a parabola at various scales.  We will also work with $\Gamma_{\frac{1}{K}}$, allowing us to bring into play some estimates for the parabola.

As we mentioned above, sharp decoupling theorems do not imply the full range of local smoothing estimates or the square function estimate.  Let us explain a little further what the issue is.  The decoupling theorem for the cone gives the following bounds, which are sharp for every $p$ between 2 and $\infty$:

\begin{equation} \label{decp<6} \| f \|_{L^p(\RR^3)} \le C_\eps R^\eps \left(\sum_{d(\theta) = R^{-1/2}} \| f_\theta \|_{L^p(\RR^3)}^2 \right)^{1/2} \textrm{   if } 2 \le p \le 6, \end{equation}

\begin{equation} \label{decp>6}  \| f \|_{L^p(\RR^3)} \le C_\eps R^{\frac{1}{4} - \frac{3}{2p} + \eps} \left(\sum_{d(\theta) = R^{-1/2}} \| f_\theta \|_{L^p(\RR^3)}^2 \right)^{1/2} \textrm{   if } p \ge 6. \end{equation}

\noindent For any given $p$, (\ref{decp>6}) implies local smoothing for that $p$.  But the inequality (\ref{decp>6}) cannot hold for any $p < 6$ because the power of $R$ would be negative.  The power of $R$ in a decoupling inequality cannot be negative because of the following example: suppose that for each $\theta$,  $| f_\theta |$ is approximately the characteristic function of $B_R$, and at each point $|f| \sim \left( \sum_\theta |f_\theta|^2 \right)^{1/2}$.  In this case, $\| f \|_{L^p} \sim \left(\sum_\theta \| f_\theta \|_{L^p}^2 \right)^{1/2}$ for all $p$.  This example is not a counterexample for local smoothing, but to prove local smoothing for some $p < 6$ we have to do better than inequality (\ref{decp<6}) in some scenarios: for instance, if the supports of $f_\theta$ are {\color{black}essentially} disjoint at time 0.   Roughly speaking, we need to improve the bound (\ref{decp<6}) when $p < 6$ and when each $f_\theta$ is {\color{black}essentially} supported on a sparse region of $B_R$.  Theorem \ref{main} makes this precise.

There are similar issues in the problem of decoupling into small caps, which was studied in \cite{DGW}.  For instance, consider an exponential sum of the form

$$f(x_1, x_2) = \sum_{j=1}^N a_j e \left( \frac{j}{N} x_1 + \frac{j^2}{N^2} x_2 \right), \textrm{ with } |a_j| \le 1 \textrm{ for all } j. \eqno{(*)}$$

\noindent The decoupling theorem for the parabola gives a sharp bound on $ \| f \|_{L^p(B_{N^2})}$ for every $p$.  But suppose we want to bound $ \| f \|_{L^p(B_R)}$ for some $R < N^2$.  If we divide the parabola into arcs $\theta$ of length $R^{-1/2}$, then each $f_\theta$ is a sum of $\sim N R^{-1/2}$ terms of $(*)$.  It's not hard to estimate the largest possible value of $\| f_\theta \|_{L^p(B_R)}$ for each $p$.  Combining this bound for $\| f_\theta \|_{L^p(B_R)}$ with decoupling gives an upper bound for $\| f \|_{L^p(B_R)}$, but it is not sharp.  When $\| f_\theta \|_{L^p(B_R)}$ is close to its largest value, then $|f_\theta|$ is concentrated on a sparse region of $B_R$.  The argument in \cite{DGW} exploits this sparsity to improve the bound from decoupling and give sharp estimates for $\| f \|_{L^p(B_R)}$ for every $p$.  The proof of the main theorem here builds on that proof.

The paper \cite{DGW} also considers a decoupling problem in which the cone is divided into small squares instead of sectors.   This problem was raised by Bourgain and Watt \cite{BW} in their work on the Gauss circle problem.  The paper \cite{DGW} shows that the square function estimate Theorem \ref{sq fun} implies a sharp estimate for this decoupling problem.

\vskip10pt

{\bf Acknowledgements.} We would like to thank Ciprian Demeter for sharing his ideas and for many helpful conversations.  He proposed the problem of decoupling into small caps and suggested improving decoupling when each $f_\theta$ is concentrated in a sparse region.  We would also like to thank Misha Rudnev for sharing thoughtful comments about \cite{GSW} which helped us in this project.  We would like to thank Terence Tao for helpful comments that improved the exposition of the proof of Proposition 3.4. We would like to thank Zhipeng Lu and Xianchang Meng for pointing out several typos in an earlier version. {\color{black} LG was supported by a Simons Investigator Award. HW was supported by the Simons Foundation grant for David Jerison.  RZ was supported by the National Science Foundation under Grant Number DMS-1856541. He would like to thank Andreas Seeger for helpful historical remarks about square functions and local smoothing. Part of this work was done when RZ was visiting MIT and he would like to thank MIT for the warm hospitality.}

We would like to thank the anonymous referees for their thorough readings and many helpful suggestions.

\section{Proof of the square function estimate from Theorem \ref{main}}

In this section, we explain how Theorem \ref{main} implies the square function estimate Theorem \ref{sq fun}, and we {\color{black} discuss} how {\color{black}the latter} implies {\color{black}the} local smoothing {\color{black}Theorem} \ref{locsmooth}.  First we recall the statement of Theorem \ref{sq fun}:

\begin{theorem*}
	For any function $f$ whose Fourier transform is supported on  $N_{R^{-1}}(\Gamma)$, we have
	$$\|f\|_{L^4({\color{black}\mathbb{R}^3})} \leq C_{\epsilon} R^{\epsilon} \|(\sum_{d(\theta)=R^{-1/2}} |f_{\theta}|^2)^{1/2}\|_{L^4{\color{black}(\mathbb{R}^3)}}.$$
\end{theorem*}
\begin{proof}
	
	
	
	Let $U$ be a translated copy of $U_{\tau,R}${\color{black}. Recall that} 
	$${\color{black}\|S_U f\|_{L^2}^2 = \int_U \sum_{\theta\subset \tau}|f_{\theta}|^2.}$$
	By Cauchy--Schwarz,
	$$  \|{\color{black}S_U f}\|_{L^2}^4 \leq |U| \int_{U} ( \sum_{\theta\subset \tau}|f_{\theta}|^2 )^2.$$
	Therefore,
	\begin{align*}
	\sum_{d(\tau)=s}~~\sum_{U\pp U_{\tau,R}} |U|^{-1} \|{\color{black}S_U f}\|_{L^2}^4
	&\leq \sum_{d(\tau)=s} \int_{{\color{black}\mathbb{R}^3}} (\sum_{\theta\subset \tau} |f_{\theta}|^2)^2\\
	&\leq \int_{{\color{black}\mathbb{R}^3}} (\sum_{\theta} |f_{\theta}|^2)^2. \qedhere
	\end{align*}
	Summing in $s$ (dyadic numbers) contributes an additional $\log R$ factor compared to Theorem~\ref{main}.

	\end{proof}

{\color{black}Essentially by} \cite{MSS}, the square function estimate in Theorem \ref{sq fun} implies {\color{black}the local smoothing Theorem \ref{locsmooth}} for the wave equation in 2+1 dimensions. This implication was sketched in Proposition 6.2 of \cite{TV}. One technical difference is that the square function considered in \cite{MSS} was the one in terms of ``small caps''  $\zeta$, $R^{-1/2}$-squares on $\Gamma$. Instead of the Littlewood--Paley estimate corresponding to equally spaced decompositions in $\mathbb{R}^2$ used in \cite{MSS} (see (1.9) and the following first two lines on page 214 of \cite{MSS}), one needs such an estimate for angular decompositions. In the $L^4$ case, such an angular square function estimate was proved by C{\'o}rdoba (see ii) on the first page of  \cite{C}). Another proof\footnote{See Proposition 4.6 in  \cite{CS}. That proposition has two parameters and C{\'o}rdoba's estimate (up to an $R^{\varepsilon}$-loss) can be viewed as a simpler one-parameter variant. See also the remark in the end of Section 4 in \cite{CS}} by Carbery--Seeger could be found in \cite{CS}.

\section{Outline of the proof of the main theorem} \label{secoutline}

In this section, we give an overview of the proof of Theorem \ref{main} and outline the rest of the paper.  First we review the statement of Theorem \ref{main}, and present it in a more detailed way.

Let $\Gamma$ be the truncated light cone $\Gamma= \{\xi_1^2+\xi_2^2= \xi_3^2, 1/2\leq |\xi_3|\leq 1\}$.  We now precisely define the sectors discussed in the introduction.
For each point $\xi \in \Gamma$ with $\xi_3 = 1$, we define a basis of $\RR^3$ as follows: the core line direction is $\mathbf{c}(\xi)=  (\xi_{1}, \xi_{2}, 1)$, the normal direction is $\mathbf{n}(\xi)=  (\xi_{1},\xi_{2}, -1),$ and the tangent direction is $\mathbf{t}(\tau)=(-\xi_{2}, \xi_{1}, 0)$.   Now for each such $\xi$, and each $s < 1$, we define the sector with direction $\xi$ and aperture $s$ as follows:




$$ \tau(s, \xi) =\{ \omega \in \RR^3: 1 \le \mathbf{c}(\xi) \cdot \omega \le 2 \textrm{ and } | \mathbf{n}(\xi) \cdot \omega| \le s^2 \textrm{ and } | \mathbf{t}(\xi) \cdot \omega | \le s \}. $$

\noindent Here $s= d(\tau)$ is the aperture of $\tau$ as described in the introduction.

For each $s$, We choose $10 s^{-1}$ evenly spaced $\xi$ in the circle $\Gamma \cap \{ \xi_3 = 1\}$, and we let $\SSS_{s}$ be the set of $\tau(s, \xi)$ for these $\xi$.  It is straightforward to check that these form a finitely overlapping cover of $N_{s^2}(\Gamma)$.

In the introduction, we considered a finitely-overlapping cover of $N_{R^{-1}} \Gamma$ by sectors $\theta$ with dimensions $\sim R^{-1} \times R^{-1/2} \times 1$.  The set of these sectors is $\SSS_{R^{-1/2}}$.

For each $\tau = \tau(s, \xi)$, and each $\rho \ge s^{-2}$, we define a box $U_{\tau,\rho}$ as follows:

\begin{equation} \label{defUtau} U_{\tau,\rho} = \{ x \in \RR^3: |\mathbf{c}(\xi) \cdot x| \le \rho s^{2} \textrm{ and } | \mathbf{n}(\xi) \cdot x| \le \rho \textrm{ and } | \mathbf{t}(\xi) \cdot {\color{black}x} | \le \rho s \}. \end{equation}

The box $U_{\tau,\rho}$ is approximately the convex hull of the union of $\theta^*$ over all sectors $\theta \subset \tau$ with $d(\theta) = \rho^{-1/2}$.  In other words, $U_{\tau,\rho}$ is approximately the smallest rectangular box such that for any $\rho^{-1/2}$-sector $\theta\subset \tau$, if a translated copy of $\theta^{*}$ intersects $U_{\tau,\rho}$, then it must lie in $10 U_{\tau,\rho}$.  We tile {\color{black}$\mathbb{R}^3$} by translated copies of $U_{\tau,\rho}$.

{\color{black} If $U$ is a translated copy of $U_{\tau,\rho}$, then we define $S_U f$ by

\begin{equation} \label{defSUf} S_U f= (\sum_{\theta \in \SSS_{\rho^{-1/2}}: \theta \subset \tau} |f_{\theta}|^2)^{1/2} |_U. \end{equation}

As written, this definition appears to depend upon $U$, $\tau$, and $\rho$.  But in fact the parameters $\rho$ and $\tau$ can be read off from $U$.  The parameter $\rho$ is the diameter of $U$.  The aperture $d(\tau) = s$ can be read off from the dimensions of $U$, which are $\rho s^2 \times \rho s \times \rho $.  And the direction $\xi$ of $\tau$ can be read off from the direction of $U$.  To illustrate this, suppose that $U$ is $B_r$ - a ball of radius $r$.  The diameter of $U$ is $r$, and so $\rho = r$.  The dimensions of $U$ are $r \times r \times r$, and so $d(\tau) = 1$.  Since $\tau$ has aperture 1, it covers all of $\Gamma$.  Therefore,

$$ S_{B_r} f = (\sum_{\theta \in \SSS_{r^{-1/2}}} |f_\theta|^2)^{1/2}  \big|_{B_r}. $$

 In particular,  $|S_{B_1}f|$ is just $|f|$ restricted to ${B_1}$.}



We define $S(r, R)$ as the smallest constant such that for every function $f$ with $\text{supp}\hat{f}\subset N_{R^{-1}}(\Gamma)$,
\begin{equation}\label{induct coeff}
\sum_{B_r\subset {\color{black}\mathbb{R}^3}} |B_r|^{-1} \|{\color{black}S_{B_r}f}\|_{L^2(B_r)}^4 \leq S(r, R) \underset{ R^{-1/2}\leq s \leq 1 }{\sum}~~\sum_{\tau \in \SSS_s}~~ \underset{U\pp U_{\tau, R}}{\sum} |U|^{-1} \|{\color{black}S_U f}\|_{L^2}^4.
\end{equation}

 On the left-hand side of  inequality~(\ref{induct coeff}), $\sum_{B_r\subset {\color{black}\mathbb{R}^3}}$ means {\color{black}the} sum over the balls $B_r$ in a finitely overlapping cover of $\mathbb{R}^3$. On the right-hand side of inequality~(\ref{induct coeff}), the first sum, $\sum_{R^{-1/2}\leq s\leq 1}$, means {\color{black}the} sum over dyadic numbers $s$ between $R^{-1/2}$ and $1$.  The last sum, $\sum_{U\pp U_{\tau, R}}$, means the sum over a set of translates of $U_{\tau, R}$ which tile $\RR^3$.


{\color{black}By H\"older's inequality, $S(r, R)<\infty$ for any $0< r, R < \infty$. We will only consider $S(r, R)$ when $r\leq R$.} Theorem \ref{main} is equivalent to the bound $S(1,R) \le C_\eps R^\eps$ since {\color{black}$|S_{B_1}f| = |f|$ on any $B_1$ and} $|f|$ is morally constant on {\color{black}$B_1$}.  We will derive Theorem \ref{main} from a series of bounds for $S(r,R)$.

In Section \ref{secincid}, we prove the Kakeya-type estimate Lemma \ref{incidenceintro}, and we use it to prove

\begin{lemma}\label{ball inflation} For {\color{black}any} $r \ge 10$, {\color{black}$r_1 \in [r, r^2]$,}
	
	$$S({\color{black}r_1}, r^2) 
	{\color{black}\leq C}.$$
	
\end{lemma}

Next we bring into play a trick from the proof of decoupling for the cone in \cite{BD}: instead of working with $\Gamma$ we work with a subset of $\Gamma$ that lies close to a short parabolic cylinder.  We let $P$ denote an arc of a parabola of length $\sim 1$ lying in $\Gamma$.  For any $K \ge 10$, we define $\Gamma_{\frac{1}{K}}$ to be the $1/K$-neighborhood of $P$ {\color{black}in $\Gamma$}.  We will eventually choose $K$ to be a large constant {\color{black}depending on $\eps$} (which remains fixed as $R \rightarrow \infty$).  The precise formula for $\Gamma_{\frac{1}{K}}$ is designed to make Lorentz rescaling work in a clean way, and we give the formula in Section \ref{lorentz} when we discuss Lorentz rescaling.  We can define a sector $\tau \subset \Gamma_{\frac{1}{K}}$ and its aperture $d(\tau)$ in the same way as before (again see Section \ref{lorentz}).  Then we define $S_K(r,R)$ as the smallest constant such that (\ref{induct coeff}) holds for every $f$ with $\text{supp } \hat f \subset N_{R^{-1}} (\Gamma_{\frac{1}{K}})$.  Since $\Gamma_{\frac{1}{K}} \subset \Gamma$, $S_K(r,R) \le S(r,R)$.  On the other hand, since $K$ will be {\color{black}a} chosen constant, $S_K(r,R)$ is almost {\color{black}equal} to $S(r,R)$ and we can use it equally well to prove Theorem \ref{main}.

If $R = K$, then $N_{R^{-1}}(\Gamma_{\frac{1}{K}})$ is the $1/K$-neighborhood of the parabolic arc $P$, and the restriction theory for the parabola can be used to study $S_K(1,K)$.  In Section \ref{secparab} we use this idea to prove the following lemma.

\begin{lemma}\label{small ball} For any $K \ge 10$, any ${\color{black}1\le } r \le K$, and any $\delta > 0$,
	$S_K(r, K) \leq C_{\delta} K^{\delta}$.
\end{lemma}

Theorem \ref{main} will follow by combining Lemma \ref{ball inflation} and Lemma \ref{small ball} with a {\color{black}Lorentz} rescaling argument.
%
    We review the Lorentz rescaling in Section \ref{lorentz}.  We use it in Section \ref{secusingrescaling} to prove the following lemma, which relates $S_K(r,R)$ for various values of $r, R$:

\begin{lemma}\label{general}
	For any $r_1< r_2 \leq r_3$,
	$$S_K(r_1, r_3)\leq \log r_2 \cdot S_K(r_1, r_2)\max_{r_2^{-1/2}\leq s \leq 1}S_K(s^2r_2,  s^2 r_3).$$
\end{lemma}

This lemma is an important motivation for working with $S_K(r,R)$.  It allows Lemma \ref{ball inflation} and Lemma \ref{small ball} to be applied at many different scales.  A key point of studying Theorem \ref{main} instead of trying to prove Theorem \ref{sq fun} directly is that it allows this multiscale analysis to come into play.




Assuming the  lemmas, we now prove bounds on $S_K(r,R)$ and use them to deduce Theorem~\ref{main}.

\begin{prop} \label{inductiveprop} For any $\eps > 0$, there exists $K = K(\eps)$ so that for any $1 \le r \le R$, we have
	
	$$ S_K(r,R) \le {\widetilde C}_\eps (R/r)^\eps. $$
\end{prop}

\begin{proof} First we note that if $r > R^{1/2}$, then Lemma \ref{ball inflation} tells us that $ S_K(r,R) \le S(r, R) \le C$, and so the conclusion holds.

Let $K=K(\epsilon) > 10$ be a constant depending only on $\epsilon$ that we will choose below.  (The constant $K(\eps)$ will depend on $\eps$ and on the constants in Lemma \ref{ball inflation} and Lemma \ref{small ball}.)

We apply induction on the ratio $R/r$.

Our base case is when $R/r\leq \sqrt{K}$.  We have already checked the proposition in case $r > R^{1/2}$.  If $r \le R^{1/2}$ and $R/r \le \sqrt{K}$, then $R \le K$.  In this case, since $K$ is a constant depending only on $\eps$, it is straightforward to check that $S_K(r, R)$ is bounded by a constant $\tilde C_K = \tilde C_\eps$.  This finishes the base case.

Next we proceed with the induction.  Given a pair $(r, R)$, our induction hypothesis is the following: for any pair $(r', R')$ with $R'/r'\leq R/2r$, we have $S_K(r', R')\leq \tilde{C}_{\epsilon}(R'/r')^{\epsilon}$.

The proof of the induction has two cases, depending on whether $r \le K^{1/2}$.

If $r\leq K^{1/2}$, we apply Lemma~\ref{general} with $r_1 = r$, $r_2 = K^{1/2}r$, and $r_3 = R$, which gives

	$$S_K(r, R) \leq  \log K \cdot S_K(r, K^{1/2}r)  \max_{r_2^{-1/2}\leq s\leq 1} S_K(s^2K^{1/2} r, s^2R). $$
	
\noindent We bound the first $S_K$ factor using Lemma \ref{small ball}, and we bound the second $S_K$ factor using induction.  These bounds give

	$$S_K(r, R) \leq  \log K \cdot S_K(r, K^{1/2}r)  \max_{r_2^{-1/2}\leq s\leq 1} S_K(s^2K^{1/2} r, s^2R) \leq \log K \cdot  C_{\delta} {\widetilde C}_{\epsilon}K^{\delta} (\frac{R}{K^{1/2} r})^{\epsilon}.$$
	
\noindent We choose $\delta = \epsilon/{\color{black}4}$, and then we choose ${\color{black}K=}K(\eps)$ large enough so that  $\log K \cdot C_{\eps/{\color{black}4}}K^{-\epsilon/{\color{black}4}} \leq 1$, and the induction closes {\color{black}in this case}.
	
Now suppose $r \ge K^{1/2}$.  Recall from the start of the proof that we may assume $r \le R^{1/2}$.  We apply Lemma \ref{general} with $r_1 = r$, $r_2 = r^2$, and $r_3 = R$, which gives

$$S_K(r, R) \leq 2\log r \cdot S_K(r, r^2) \max_{r^{-1}\leq s \leq 1} S_K(s^2 r^2, s^2 R). $$

We bound the first $S_K$ factor using Lemma \ref{ball inflation} and we bound the second $S_K$ factor using induction, giving
	
	$$S_K(r, R) \leq 2\log r \cdot S_K(r, r^2) \max_{r^{-1}\leq s \leq 1} S_K(s^2 r^2, s^2 R) \leq 2\log r \cdot {\color{black}C} {\widetilde C}_{\epsilon}
	(\frac{R}{r^2})^{\epsilon}.$$

\noindent We  choose $K = K(\eps)$ large enough so that for all $r \ge K^{1/2}$, we have $2\log r\cdot {\color{black}C} r^{-{\color{black} \epsilon}} \leq 1$, and the induction closes {\color{black}in this case}.
\end{proof}

Finally we show how Proposition \ref{inductiveprop} implies Theorem \ref{main}.

\begin{proof}
 Proposition \ref{inductiveprop} implies that for every $\eps > 0$, we can choose $K = K(\eps)$ so that $S_K(1, R) \le C_\eps R^\eps$ for all $R$.  Suppose that the support of $\hat f$ is contained in $N_{R^{-1}}(\Gamma_{\frac{1}{K}}) \subset B_3$.  Since $|f|$ is morally constant on unit balls, we have\footnote{Strictly speaking, one need to apply Lemma~\ref{lem: locally constant} and Lemma~\ref{lem: convolution} to justify the first ``$\lesssim$'' in inequality~(\ref{SKmain}). This is similar to the arguments in Section~\ref{secparab} where we do in full details.}

\begin{equation} \label{SKmain} \int_{{\color{black}\mathbb{R}^3}} |f|^4 \lesssim \sum_{B_1 \subset {\color{black}\mathbb{R}^3}} \| f \|_{L^2(B_1)}^4 {\color{black}=\sum_{B_1 \subset \mathbb{R}^3} \|S_{B_1}f\|_{L^2(B_1)}^4} \le C_\eps R^\eps \underset{ R^{-1/2}\leq s \leq 1 }{\sum}~~\sum_{d(\tau)=s}~~ \underset{U\pp U_{\tau,R}}{\sum} |U|^{-1} \|{\color{black}S_U f}\|_{L^2}^4. \end{equation}

This inequality is essentially Theorem \ref{main} except that we assumed that $\hat f$ is supported on $N_{R^{-1}}(\Gamma_{\frac{1}{K}})$ instead of $N_{R^{-1}}(\Gamma)$.  Since $N_{R^{-1}}(\Gamma)$ can be covered by $O(K) = O_\eps(1)$ affine copies of $\Gamma_{\frac{1}{K}}$, we can reduce Theorem \ref{main} to (\ref{SKmain}).  Here are the details.

Take $\{A_j\}_{1\leq j\lesssim K}$ to be a collection of linear transformations such that $\Gamma \subset \bigcup A_j (\Gamma_{\frac{1}{K}})$. Here each $A_j$ is a composition of a scaling by a factor $\sim 1$ and a rotation in the $(\xi_1, \xi_2)-$plane\footnote{One can choose  $\lesssim 1$ rotations $R_k$  such that $\bigcup_{k} R_k(\Gamma_{\frac{1}{K}})$ covers  $\Gamma(h)= \Gamma\cap \{ h\leq \xi_3\leq h+K/10\}$ for some $h\sim 1$.   Then we choose $\lesssim K$ dilations $D_l$ such that $\Gamma\subset \bigcup_{l} D_l (\Gamma(h))$. We define $A_j = D_l R_k$. for some $l$ and $k$.}. Similarly, we can arrange that  $N_{R^{-1}}(\Gamma)\subset \bigcup A_j \big( N_{R^{-1}}(\Gamma_{\frac{1}{K}}) \big)$.
	Let ${\color{black}\{\psi_j\}}$ be a ${\color{black}C^{\infty}}$ partition of unity subordinate to this covering. {\color{black}This partition of unity only depends on $K$.}
	If $f$ is a function  whose Fourier transform is supported on $N_{R^{-1}}(\Gamma)$, then $\hat f=\sum_j  \psi_j \hat f$.   Define $f_j$ by $\hat f_j = \psi_j \hat f$  {\color{black}and ${\hat f_{j,\theta}} = \psi_j {\hat f_{\theta}}$}.  The support of $\hat f_j$ is contained in $A_j (N_{R^{-1}}(\Gamma_{\frac{1}{K}}))$.  Since (\ref{SKmain}) is invariant under rotations and approximately invariant under rescaling by a factor $\sim 1$, (\ref{SKmain}) holds for each function $f_j$.

	Now {\color{black}by} the  triangle inequality and H\"{o}lder's inequality,
	\begin{align*}
	\|f\|_{L^4({\color{black}\mathbb{R}^3})}^4 & \lesssim K^3 \sum_j \|f_j\|_{L^4({\color{black}\mathbb{R}^3})}^4 \\
	&\lesssim K^3 C_{\epsilon} R^{\epsilon} \sum_j \underset{R^{-1/2}\leq s\leq 1}{\sum}~~\sum_{d(\tau)=s} ~~\underset{U\pp U_{\tau,R}}{\sum} |U|^{-1} \|{\color{black}S_U f_{j}}\|_{L^2}^4\\
	&\lesssim K^3 C_{\epsilon} R^{\epsilon} \underset{R^{-1/2}\leq s\leq 1}{\sum}~~
	\sum_{d(\tau)=s}~~\underset{U\pp U_{\tau,R}}{\sum} |U|^{-1} ( \sum_j\|{\color{black}S_U f_{j}}\|_{L^2}^2)^2\\
	&\lesssim_{{\color{black}K}} C_{\epsilon} R^{\epsilon} \underset{R^{-1/2}\leq s\leq 1}{\sum}~~\sum_{d(\tau)=s} \underset{U\pp U_{\tau,R}}{\sum} |U|^{-1} \|{\color{black}S_U f}\|_{L^2}^4.
	\end{align*}
	
	{\color{black}To see the last inequality, note that $f_{j, {\color{black}\theta}} = f_{{\color{black}\theta}} * {\check \psi_j}$ and ${\check \psi_j}$ is rapidly decaying outside the ball of radius $K$  centered at the origin. Hence, by Lemma~\ref{lem: convolution},  each $\|f_{j, {\color{black}\theta}}\|_{L^2 (B_1)} \lesssim_K \|f_{{\color{black}\theta}}\|_{L^2 (w_{B_1, E})}$ for any polynomially decaying weight $w_{B_1, E}$. It suffices to take $E$ large enough.}
	
	Since $K$ is a constant only depending on $\epsilon$, this gives Theorem \ref{main}. \end{proof}

\section{ A Kakeya-type estimate} \label{secincid}
In this section, we prove the Kakeya-type estimate Lemma \ref{incidenceintro}, and we use it to prove Lemma \ref{ball inflation}.  First we recall the statement.

\begin{lemma*} Suppose that $\hat f$ has support on $N_{r^{-2}}(\Gamma)$.  Let $g$ denote the {\color{black}(squared)} square function $g=\sum_{\theta \in \SSS_{r^{-1}}} |f_{\theta}|^2$.  Then
	$$\int_{\color{black} \mathbb{R}^3} |g|^2 \lesssim \sum_{R^{-1/2} \leq s\leq 1}   \sum_{\tau \in \SSS_s} \sum_{U\pp U_{\tau,R}} {\color{black} |U|^{-1}} \|{\color{black}S_U f}\|_{L^2}^4.$$
	
\end{lemma*}

(Comparing with the statement in the introduction, we use $r^2$ in place of $R$.  This makes the algebra in the proof a little simpler, and it connects with the notation in Lemma \ref{ball inflation}.)


\begin{proof}[Proof of Lemma \ref{incidenceintro}]
Suppose that $\text{supp } \hat f\subset N_{r^{-2}}(\Gamma)$.  Recall that

$$ g=\sum_{\theta \in \SSS_{r^{-1}}}|f_{\theta}|^2. $$

The Fourier transform of $|f_{\theta}|^2$ is supported on the Minkowski sum $\tilde{\theta}=\theta + (-\theta)$.  The set $\tilde{\theta}$ is itself a plank of dimensions $\sim r^{-2}\times r^{-1}\times 1$ centered at the origin.  Notice that while the original sectors $\theta$ are disjoint, the planks $\tilde \theta$ are not disjoint.  The way that they overlap plays an important role in the proof.

%

The Minkowski sum $\tilde \theta (\xi) = \theta(\xi) + (- \theta(\xi))$ is approximately equal to the following rectangular box:
$$ \tilde \theta(\xi) \approx
  \{ \omega \in \RR^3:  |\mathbf{c}(\xi) \cdot \omega| \le 1 \textrm{ and } | \mathbf{n}(\xi) \cdot \omega| \le r^{-2} \textrm{ and } | \mathbf{t}(\xi) \cdot \omega | \le r^{-1} \}, $$
  where two convex sets $A\approx B$ means that $A\subset 10 B\subset 100A$.

The overlapping of the boxes $\tilde \theta$ is best described in terms of similar rectangular boxes at smaller scales.  For any {\color{black}dyadic} $\sigma$ in the range $r^{-1} \le \sigma \le 1$, and any $\xi$ as above, we define a box $\Theta = \Theta(\sigma, \xi)$ by

\begin{equation} \label{defTheta}
 \Theta(\sigma, \xi) =\{\omega:  | \mathbf{c}(\xi)\cdot \omega |\leq \sigma^2 \textrm{ and }| \mathbf{n}(\xi)\cdot \omega|\leq r^{-2} \textrm{ and } | \mathbf{t}(\xi)\cdot \omega| \leq  r^{-1} \sigma\}.
\end{equation}

\noindent Notice that $\Theta(1, \xi)$ is equal to $\tilde \theta(\xi)$, and for $\sigma < 1$, $\Theta(\sigma, \xi) \subset \tilde \theta(\xi)$.  At the other extreme, $\Theta(r^{-1}, \xi)$ is essentially the ball of radius $r^{-2}$ centered at the origin, regardless of $\xi$.

If we intersect $\Theta(\sigma, \xi)$ with the slab  $ \{ (1/2) \sigma^2 \le \omega_3 \le \sigma^2 \}$, then it lies in the $r^{-2}$-neighborhood of the light cone.  Let $\Gamma(\sigma^2)$ denote the part of the light cone where $(1/2) \sigma^2 \le \omega_3 \le \sigma^2$.  Each $\Theta(\sigma, \xi) \cap \{ (1/2) \sigma^2 \le \omega_3 \le \sigma^2 \}$ is a sector of $N_{r^{-2}}(\Gamma(\sigma^2))$, just as $\theta$ is a sector of $N_{r^{-2}}(\Gamma)$.  The number of such sectors needed to cover $N_{r^{-2}}(\Gamma(\sigma^2))$ is $\sim \sigma r$.  If $| \xi - \xi'| > \sigma^{-1} r^{-1}$, then $\Theta(\sigma, \xi) \cap \Theta (\sigma, \xi') \cap  \{ (1/2) \sigma^2 \le \omega_3 \le \sigma^2 \}$ is empty.  Conversely, if $| \xi - \xi' | < \sigma^{-1} r^{-1}$ then $\Theta(\sigma, \xi) \cap  \{ (1/2) \sigma^2 \le \omega_3 \le \sigma^2 \}$ is comparable to $\Theta (\sigma, \xi') \cap  \{ (1/2) \sigma^2 \le \omega_3 \le \sigma^2 \}$.  By symmetry, the same holds when we intersect with $\{ - \sigma^2 \le \omega_3 \le - (1/2) \sigma^2 \}$ at the other side of the light cone.  Now by convexity, we conclude that if $| \xi - \xi'| \le \sigma^{-1} r^{-1}$, then $\Theta(\sigma, \xi) \subset 2 \Theta(\sigma, \xi')$.

For each dyadic $\sigma$ in the range $r^{-1} \le \sigma \le 1$, let $\CP_\sigma$ be a set of $\sim \sigma r$ planks of the form $\Theta(\sigma, \xi)$ with the directions $\xi$ evenly spaced in the circle.  (The letters $\CP$ stand for centered plank.) The size of $\CP_\sigma$ is chosen so that for any $\Theta(\sigma, \xi)$, we can choose $\Theta(\sigma, \xi') \in \CP_\sigma$ so that $\Theta(\sigma, \xi) \subset 2 \Theta(\sigma, \xi')$.  We define $\CP$ as a union over dyadic scales: $ \CP = \cup_{r^{-1} \le \sigma \le 1} \CP_\sigma. $  Since $\Theta(1, \xi)$ is the same as $\tilde \theta(\xi)$, $\CP_1 = \SSS_{r^{-1}}$.  On the other hand, $\CP_{r^{-1}}$ is a set with one element, which is essentially the ball of radius $r^{-2}$ around the origin.

For a given $\theta(\xi)$ and a given scale $\sigma$, there are $\sim 1$ ${\color{black}\Theta =} \Theta(\sigma, \xi') \in \CP_\sigma$ with $\Theta \subset 2 \tilde \theta$.  To see this, note on the one hand that $\Theta(\sigma, \xi) \subset \tilde \theta(\xi)$, and we can choose $\Theta(\sigma, \xi') \in \CP_\sigma$ so that $\Theta(\sigma, \xi') \subset 2 \Theta (\sigma, {\color{black}\xi})$.  On the other hand, $\tilde \theta (\xi) \cap N_{r^{-2}}(\Gamma(\sigma^2))$ is essentially equal to the sector $\Theta(\sigma, \xi) \cap \{ (1/2) \sigma^2 \le \omega_3 \le \sigma^2 \}$, and so $2 \tilde \theta (\xi)$ contains $\Theta (\sigma, \xi')$ only if $| \xi - \xi' | \lesssim \sigma^{-1} r^{-1}$.

In our proof, $r$ remains fixed but we have to consider various scales $\sigma$.  To simplify notation, we abbreviate $\SSS_{r^{-1}}$ as $\SSS$.  Now for each scale $\sigma$, for each $\theta = \theta(\xi) \in \SSS = \SSS_{r^{-1}}$, we associate one $\Theta = \Theta(\sigma, \xi') \in \CP_\sigma$ with $|\xi' - \xi| \le \sigma^{-1} r^{-1}$.  For each $\Theta \in \CP_\sigma$, we let $\SSS_\Theta$ be the set of all $\theta \in \SSS$ which are associated with $\Theta$.  So for each $\sigma$, $\SSS= \bigsqcup_{\Theta \in \CP_\sigma} \SSS_\Theta$.  If $\theta \in \SSS_\Theta$, then $\Theta \subset 2 \tilde \theta$.

Let $\Omega = \cup_{\theta \in \SSS} \tilde \theta \sim \cup_{\Theta \in \CP_1} \Theta$.  Since $(|f_\theta|^2)^\wedge$ is supported on $\tilde \theta$, it follows that $\hat g$ is supported on $\Omega$.   We break $\Omega$ into pieces associated with different scales $\sigma$ as follows.  We define $\Omega_{\le \sigma} = \cup_{\Theta \in \CP_\sigma} \Theta$.  Then we define $\Omega_\sigma = \Omega_{\le \sigma} \setminus \Omega_{\le \sigma/2}$ if $\sigma > r^{-1}$, and we define $\Omega_{r^{-1}} = \Omega_{\le r^{-1}}$, so that

$$\Omega = \bigsqcup_{r^{-1} \le \sigma \le 1} \Omega_\sigma. $$

\noindent (Here $\bigsqcup$ denotes a disjoint union, and the union is over dyadic $\sigma$.)

Now if $\omega \in \Omega_\sigma$, we bound $|\hat g(\omega)|$ as follows:

\begin{equation}\label{decgomega} | \hat g(\omega) | =  |\sum_{\theta \in \SSS} ( |f_\theta|^2)^\wedge (\omega)| \le \sum_{\Theta \in \CP_\sigma} | \sum_{\theta \in \SSS_\Theta}  ( |f_\theta|^2)^\wedge (\omega)  | . \end{equation}

\begin{lemma} \label{lemTheta1} If $\Theta \in \CP_\sigma$ makes a non-zero contribution to the right-hand side of (\ref{decgomega}) for an $\omega\in \Omega_{\sigma}$, then $\omega \in 4 \Theta$.
\end{lemma}

\begin{proof} Suppose that $ \sum_{\theta \in \SSS_\Theta}  ( |f_\theta|^2)^\wedge (\omega) $ is non-zero.  Then we must have $\omega \in \tilde \theta$ for some $\theta \in \SSS_\Theta$.
Suppose $\theta = \theta(\xi)$ and $\Theta = \Theta(\sigma, \xi')$.  Since $\theta \in \SSS_{\Theta}$, we know that $| \xi - \xi'| \le \sigma^{-1} r^{-1}$ and so $\Theta(\sigma, \xi) \subset 2 \Theta$.

We claim that $\tilde \theta \cap \Omega_{\le \sigma}$ is contained in $2 \Theta(\sigma, \xi)$.  This will finish the proof, because $\omega \in \tilde \theta \cap \Omega_{\le \sigma} \subset 2 \Theta(\sigma, \xi) \subset 4 \Theta (\sigma, \xi')$.

To check the claim, we have to understand the geometry of the set $\Omega_{\le \sigma}$.   To picture the set $\Omega_{\le \sigma}$, we found it helpful to consider the intersection of $\Theta(\sigma, \xi)$ with the plane $\omega_3 = h$.  We assume $|h| \le \sigma^2$ -- otherwise the intersection is empty.  The intersection $\Theta(\sigma, \xi) \cap \{ \omega_3 = h \}$ is a rectangle with dimensions $r^{-1} \sigma \times {\color{black}\sqrt{2}} r^{-2}$, and the long side of the rectangle is tangent to the circle of radius $h$ around the origin at the point $h \xi$.  Therefore, $\Theta(\sigma, \xi) \cap \{ \omega_3 = h\}$ is contained in the annulus $ \{h^2 \le \omega_1^2 + \omega_2^2 \le h^2 + r^{-2} \sigma^2\}$.  If we rotate $\xi$, the rectangle $\Theta(\sigma, \xi) \cap \{ \omega_3 = h\}$ rotates also, and the union of these rotated rectangles over all $\xi$ is equal to this annulus.  Therefore, if $h \le \sigma^2$, $\Omega_{\le \sigma} \cap \{ \omega_3 = h \}$ is approximately equal to this annulus:

\begin{equation}\label{Omega<sig}  \Omega_{\le \sigma} \cap \{ \omega_3 = h \} \sim  \{ \omega: \omega_3 = h, h^2 \le \omega_1^2 + \omega_2^2 \le h^2 + r^{-2} \sigma^2\}. \end{equation}


On the other hand, $\tilde \theta(\xi) \cap \{ \omega_3 = h \} = \Theta(1, \xi) \cap \{\omega_3 = h \}$ is a rectangle of dimensions $\sim r^{-1} \times r^{-2}$ which is tangent to the circle of radius $h$ at $h \xi$.  The intersection of  this rectangle with the annulus above is contained in a shorter rectangle with the same center and with dimensions $\sigma r^{-1} \times r^{-2}$, which in turn is contained in $2 \Theta(\sigma, \xi) \cap \{ \omega_3 = h \}$.  Since this holds for every $h$ with $|h| \le \sigma^2$, we see that $\tilde \theta(\xi) \cap \Omega_{\le \sigma} \subset 2 \Theta(\sigma, \xi)$ as claimed.
\end{proof}

Using Lemma \ref{lemTheta1}, we can rewrite inequality (\ref{decgomega}): if $\omega \in \Omega_{\sigma}$, then

\begin{equation} \label{decgomega2} | \hat g(\omega) |  \le \sum_{\Theta \in \CP_\sigma, \omega \in 4 \Theta} | \sum_{\theta \in \SSS_\Theta}  ( |f_\theta|^2)^\wedge (\omega)  | . \end{equation}

\begin{lemma} \label{lemTheta2} For any $\omega \in \Omega_\sigma$, the number of $\Theta \in \CP_\sigma$ so that $\omega \in 4 \Theta$ is bounded by a constant $C$.
\end{lemma}

\begin{proof} Building on the description of $\Omega_{\le \sigma}$ in (\ref{Omega<sig}) above, we see that if $|h| \le  \sigma^2/4$, then $\Omega_\sigma \cap \{ \omega_3 = h \}$ is approximately given by
		
\begin{equation}\label{annulus'} \{ h^2 + (1/4) r^{-2} \sigma^2 \le \omega_1^2 + \omega_2^2 \le h^2 + r^{-2} \sigma^2 \}. \end{equation}

If $ \sigma^2/4 \le |h| \le \sigma^2$, then $\Omega_\sigma \cap \{ \omega_3 = h \}$ is approximately given by

\begin{equation}\label{annulus''}  \{ h^2 \le \omega_1^2 + \omega_2^2 \le h^2 + r^{-2} \sigma^2 \}. \end{equation}

Let $C_{h, \rho}$ be the circle defined by $\omega_3 = h$ and $\omega_1^2 + \omega_2^2 = \rho^2$ with $|h| \le \sigma^2$ and $\rho$ chosen {\color{black}such that $C_{h, \rho}$ lies} in (\ref{annulus'}) or (\ref{annulus''}).  These circles cover $\Omega_\sigma$.
For any $\xi$, we will compute in the next two paragraphs that the fraction of {\color{black}$C_{h, \rho}$} contained in $4 \Theta(\sigma, \xi)$ is $\lesssim \sigma^{-1} r^{-1}$.  There are $\sim \sigma r$ different $\Theta(\sigma, \xi) \subset \CP_\sigma$.  By circular symmetry, each frequency $\omega \in C_{h, \rho}$ lies in $4 \Theta$ for approximately the same number of $\Theta \in \CP_\sigma$, and so each frequency $\omega$ lies in $4 \Theta$ for $\le C$ different $\Theta \in \CP_\sigma$.

We first do the case $|h| \le  \sigma^2/4$.  Recall that $\Theta(\sigma, \xi) \cap \{ \omega_3 =h \}$ is a rectangle with dimensions $r^{-1} \sigma \times r^{-2}$ which is tangent to the circle of radius $|h|$.  Suppose for now that $r^{-1} \sigma \le |h|$.  If $A, B$ are the two endpoints of this rectangle and $O$ is the origin, then the angle $AOB$ is approximately $r^{-1} \sigma / |h|$.  The angle between the rectangle $\Theta \cap \{\omega_3=h\}$ and the circle $C_{h,\rho}$ is approximately equal to the angle $AOB$.  Therefore, the arc length of $4 \Theta \cap C_{h, \rho}$ is bounded by

$$ \textrm{Length }(4 \Theta \cap C_{h, \rho} ) \lesssim r^{-1} \sigma^{-1} |h|. $$

\noindent Since the length of $C_{h, \rho}$ is $2\pi \rho \sim |h|$, the fraction of $C_{h, \rho}$ contained in $4 \Theta$ is $\lesssim r^{-1} \sigma^{-1}$ as desired.

If $|h| {\color{black}<} r^{-1} {\color{black}\sigma}$, then the angle $AOB$ is $\sim 1$, and the length of $4 \Theta \cap C_{h, \rho}$ is approximately $r^{-2}$.  In this case the length of $C_{h, \rho}$ is $2\pi \rho \sim r^{-1} \sigma$, and so the fraction of $C_{h, \rho}$ covered by $4 \Theta$ is still $\lesssim r^{-1} \sigma^{-1}$.
	
Finally, suppose that $\sigma^2/4\leq |h| \leq \sigma^2$.  In this case $4 \Theta 	\cap C_{h, \rho}$ has arc length $\sim \sigma r^{-1}$ (the long side of the rectangle $\Theta \cap \{ \omega_3 = h \}$.  Since the length of $C_{h, \rho}$ is  $2\pi\rho \sim |h| \sim \sigma^2$, the fraction of $C_{h, \rho}$ covered by $4 \Theta$ is again $\lesssim \sigma^{-1} r^{-1}$.
\end{proof}

Remark.  If $\omega \in \Omega_\sigma$ and $|\omega_3|$ is much smaller than $\sigma^2$, then $\omega$ lies in two rather different $\Theta \in \CP_\sigma$, and maybe also on other $\Theta$ neighboring these two.  This is because a point outside a circle lies on two lines tangent to the circle.

\vskip5pt

Applying Cauchy--Schwarz to (\ref{decgomega2}) and using Lemma \ref{lemTheta2} we see that if $\omega \in \Omega_\sigma$, then

\begin{equation} \label{decgomega3} | \hat g(\omega) |^2  \lesssim \sum_{\Theta \in \CP_\sigma, \omega \in 4 \Theta} | \sum_{\theta \in \SSS_\Theta}  ( |f_\theta|^2)^\wedge (\omega)  |^2 . \end{equation}

We let $\eta_\Theta$ be a smooth function which is {\color{black}$\geq 1$} on $4 \Theta$ and decays {\color{black}rapidly} outside $4 \Theta$.  Summing over all dyadic $\sigma$, we see that for every frequency $\omega$,

$$ | \hat g(\omega)|^2 \lesssim \sum_{\Theta \in \CP} \left|  \eta_\Theta(\omega) \sum_{\theta \in \SSS_\Theta}  ( |f_\theta|^2)^\wedge (\omega)\right|^2. $$

Now we integrate and use Plancherel, giving

$$ \int |g|^2 \lesssim \sum_{\Theta \in \CP} \int  |  \eta_\Theta^\vee * \sum_{\theta \in \SSS_{{\color{black}\Theta}}}  |f_\theta|^2 |^2. $$

Now we can choose $\eta_\Theta$ so that $| \eta_\Theta^\vee (x)| \lesssim |\Theta^*|^{-1}$ for all $x$, and $\eta_\Theta^\vee$ is supported on $\Theta^*$.  Therefore, it is natural to break up the right integral into translated copies of $\Theta^*$:

$$ \int |g|^2 \lesssim \sum_{\Theta \in \CP}  \sum_{U \pp \Theta^*} \int_U  |  \eta_\Theta^\vee * \sum_{\theta \in \SSS_{{\color{black}\Theta}}}  |f_\theta|^2 |^2. $$

In the last integral, for each $x \in U$, we have

$$  |  \eta_\Theta^\vee * \sum_{\theta \in \SSS_\Theta} |f_\theta|^2(x) | \lesssim |U|^{-1} \int \eta_U \sum_{\theta \in \SSS_{{\color{black}\Theta}}}  |f_\theta|^2 , $$

\noindent where $\eta_U(z)=|\Theta^*| \cdot \max_{y\in z+ \Theta^* - U} |\eta_{\Theta}^{\vee}(y)|$ is a bump function with $\|\eta_U\|_{\infty} \sim 1$ supported on $2U$. We remark that the arguments presented here exploit the locally constant property. We shall discuss another variant of this property in Lemma~\ref{lem: locally constant}.

Therefore,

$$ \int |g|^2 \lesssim \sum_{\Theta \in \CP}  \sum_{U \pp \Theta^*}  |U|^{-1} \left(  \int \eta_U \sum_{\theta \in \SSS_\Theta}  |f_\theta|^2 \right)^2.  $$

We associate $\Theta(\sigma, \xi)$ to $\tau(\sigma^{-1} r^{-1}, \xi)$.  This gives a bijection from $\CP_\sigma$ to $\SSS_s$ with $s = \sigma^{-1} r^{-1}$.  If $\Theta(\sigma, \xi) \subset 2 \tilde\theta(\xi')$, then we saw above that $| \xi - \xi'| \lesssim \sigma^{-1} r^{-1}$, and so $\theta(\xi') \subset 4 \tau(\sigma^{-1} r^{-1}, \xi)$.   In particular, if $\theta \in \SSS_\Theta$, then $\theta \subset 4 \tau$.  Also $\Theta(\sigma, \xi)^*$ is comparable to $U_{\tau(\sigma^{-1} r^{-1}, \xi), r^2}$, which we can see by comparing the definition of $U_{\tau,r^2}$ in (\ref{defUtau}) with the definition of $\Theta$ in (\ref{defTheta}).   Rewriting the last inequality in terms of $\tau \in \SSS_s$ instead of $\Theta \in \CP_\sigma$, we get

$$ \int |g|^2 \lesssim \sum_{r^{-1} \le s \le 1} \sum_{\tau \in \SSS_s} \sum_{U \pp U_{\tau,r^2}} |U|^{-1}  \left(\int  \eta_U \sum_{\theta \subset \tau}  |f_\theta|^2 \right)^2. $$


{\color{black}By the definition of $S_U f$,}
$$ \sum_{U\pp U_{\tau, r^2}} (\int \eta_U \sum_{\theta \subset \tau} | f_\theta|^2)^2  \lesssim \sum_{U\pp U_{\tau,r^2} } \| S_{U} f \|_{L^2}^4.$$

Plugging this in, we get

$$ \int |g|^2 \lesssim
\underset{r^{-1}\leq s \leq 1}{\sum} ~ \underset{d(\tau)=s}{\sum}  ~ \underset{U\pp U_{\tau,r^2}}{\sum} |U|^{-1} \|S_Uf\|_{L^2}^4.
 $$

This proves Lemma \ref{incidenceintro} by taking $r = R^{\frac{1}{2}}$. \end{proof}

We use this Kakeya-type estimate as well as local orthogonality to prove Lemma \ref{ball inflation}.  First we recall local orthogonality, and then we recall the statement of Lemma \ref{ball inflation}.

Local orthogonality is written using a weight functions localized a given ball.  For a ball $B_R$ of radius $R$, define the weight

$$w_{B_R, E} (x) = (1+\frac{\mathrm{dist} (x, B_R)}{R})^{-E}.$$

\begin{lemma}[Local $L^2$ orthogonality lemma, essentially Proposition 6.1 in \cite{BD2}]\label{localorthogonality} Suppose that $f \in L^2(\RR^n)$.  Suppose that $f = \sum_\theta f_\theta$, where  $\mathrm{supp} {\hat f_{\theta}} \subset \theta$ in the Fourier space.  In this statement the sets $\theta$ are arbitrary.  Suppose that $r > 0$ and that each $\xi \in \RR^n$ lies in $N_{r^{-1}(\theta)}$ for at most $M$ different sets $\theta$ appearing in the sum.  Then for any $E>0$,
    $$\|f\|_{L^2 (B_r)}^2 \lesssim_{M, E} \sum_{\theta\in\mathcal{I}} \|f_{\theta}\|_{L^2 (w_{B_r, E})}^2.$$
\end{lemma}

To prove Lemma~\ref{localorthogonality}, it suffices to take a function $\psi_{B_r}$ such that $ \psi_{B_r}\gtrsim 1 $ on $B_r$,  $|\psi_{B_r}(x)|\leq C_E (1+r^{-1}\text{dist}(x, B_r))^{-E/2}$, and $\hat{\psi}_{B_r}\subset B(0, r^{-1})$. Then $\|f\|_{L^2(B_r)}\lesssim \|f\psi_{B_r} \|_{L^2}$. We apply Plancherel's theorem  and observe that  the support of $\widehat{f}_{\theta}\ast \widehat{\psi}_{B_r}$  lies in $N_{r^{-1}}(\theta)$.

Now we turn to the proof of Lemma \ref{ball inflation}.  Unwinding the definition of $S(r,R)$, Lemma \ref{ball inflation} says

\begin{lemma*} 
{\color{black}If} $\hat f$ is supported on $N_{r^{-2}}(\Gamma)$ {\color{black}and $r_1 \in [r, r^2]$}, then 
\begin{equation} \label{eqballinfl}	\sum_{{\color{black}B_{r_1}}\subset {\color{black}\mathbb{R}^3}} |{\color{black}B_{r_1}}|^{-1} \|{\color{black}S_{B_{r_1}} f}\|_{L^2({\color{black}B_{r_1}})}^4
{\color{black}\lesssim} \underset{ r^{-1} \leq s \leq 1 }{\sum}~~\sum_{d(\tau)=s}~~ \underset{U\pp U_{\tau, r^2}}{\sum} |U|^{-1} \|{\color{black}S_U f}\|_{L^2}^4.  \end{equation}
\end{lemma*}

\begin{proof} [Proof of Lemma \ref{ball inflation}] As in Lemma \ref{incidenceintro}, let $g = \sum_{\theta \in \SSS_{r^{-1}}} |f_\theta|^2$.   The functions $f_\theta$ have  essentially disjoint Fourier support. Since $r\leq r_1$, each point $\xi$ lies in $\lesssim 1$ many $N_{r_1^{-1}}(\theta)$.
	
We choose $E$ sufficiently large (for instance $E = 10$).  Then we apply the local $L^2$ orthogonality Lemma \ref{localorthogonality}, on each ${\color{black}B_{r_1}}$:
	
	$$ \| {\color{black}S_{B_{r_1}} f} \|_{L^2({\color{black}B_{r_1}})}^2 {\color{black}=\int_{B_{r_1}} \sum_{d(\tau) = r_1^{-1/2}} |f_{\tau}|^2} \lesssim {\color{black}\int_{\mathbb{R}^3} w_{B_{r_1}, E}\cdot} {\color{black}\sum_{d(\tau) = r_1^{-1/2}} \sum_{\theta \subset \tau} |f_{\theta}|^2} 
	\sim \int_{{\color{black}\mathbb{R}^3}} {\color{black}w_{B_{r_1}, E}\cdot}g. $$
	
	By Cauchy--Schwarz, we get
	
	$$ |{\color{black}B_{r_1}}|^{-1} \| {\color{black}S_{B_{r_1}} f} \|_{L^2({\color{black}B_{r_1}})}^4 \lesssim \int_{{\color{black}\mathbb{R}^3}} {\color{black}w_{B_{r_1}, E/2}\cdot}|g|^2. $$
	
Summing over $B_{r_1}$,

$$ 	\sum_{{\color{black}B_{r_1}}\subset {\color{black}\mathbb{R}^3}} |{\color{black}B_{r_1}}|^{-1} \|{\color{black}S_{B_{r_1}} f}\|_{L^2({\color{black}B_{r_1}})}^4 \lesssim \int_{{\color{black}\mathbb{R}^3}} |g|^2. $$

Lemma \ref{incidenceintro} bounds $\int_{{\color{black}\mathbb{R}^3}} |g|^2$ by the right-hand side of (\ref{eqballinfl}).
\end{proof}

{\color{black}

\section{The Lorentz rescaling}\label{lorentz}

Lorentz transformations are the symmetries of our problem, and they have been used in many earlier papers on this topic (cf. for instance \cite{W} and \cite{BD}).  Here we review {\color{black}the Lorentz} rescaling and check the properties that we will need in our rescaling argument in the next {\color{black}two sections}.

The {\color{black}piece $\Gamma_{\frac{1}{K}}$} is defined to work well with Lorentz transformations, and we now record the formula.  This formula and the Lorentz rescaling generally look nicest in a rotated coordinate system where the light cone is given by the equation $2 {\color{black} \nu_1\nu_3} = {\color{black} \nu_2^2}$.  Here ${\color{black} \nu_2} = \xi_1$, ${\color{black} \nu_1} = 2^{-1/2} (\xi_3 - \xi_2)$ and ${\color{black} \nu_3} = 2^{-1/2}(\xi_3 + \xi_2)$.  
In these coordinates, if we intersect the light cone with the plane ${\color{black} \nu_3}=1$ then we get the parabola ${\color{black} \nu_1} = (1/2) {\color{black} \nu_2^2}$.  So the light cone is actually the cone over a parabola.

Now $\Gamma_{\frac{1}{K}}$ is defined as follows.

$$\Gamma_{\frac{1}{K}}=\{ 2{\color{black} \nu_1 \nu_3}={\color{black} \nu_2^2}, 1-\frac{1}{K}\leq {\color{black} \nu_3}\leq 1,  |\frac{{\color{black} \nu_2}}{{\color{black} \nu_3}}|\leq 1\}.$$

For any real number $\eta$ with $|\eta| < 1$ and $0 < s < 1$ satisfying $-1 \leq {\color{black}\eta} \pm s \leq 1$, we can define a \emph{surface sector} $\Lambda \subseteq \Gamma_{\frac{1}{K}}$ by

\begin{equation}\label{defnoftauandtheta}
\Lambda = \Lambda (\eta, s) = \{({\color{black} \nu_1}, {\color{black} \nu_2}, {\color{black} \nu_3}) \in \Gamma_{\frac{1}{K}}: |\frac{{\color{black} \nu_2}}{{\color{black} \nu_3}}-{\color{black}\eta}| < s \}.
\end{equation}

\noindent Here $s$ is the {\it aperture} of $\Lambda$, also denoted by $d(\Lambda)$. For each $\Lambda$, let $\eta(\Lambda)$ denote the $\eta$ in \eqref{defnoftauandtheta}.

Each surface sector $\Lambda$ is closely associated to a sector $\tau = \tau(\Lambda)$, which is a rectangular box containing $\Lambda$ with smallest comparable dimensions.  The sector $\tau(\Lambda)$ is approximately the convex hull of $\Lambda$ in the sense that $\frac{1}{10} \tau(\Lambda)\subset \text{ConvexHull}(\Lambda) \subset 10\tau(\Lambda)$.  Similarly, starting with any sector $\tau$, there is an associated surface sector $\Lambda_\tau = \tau \cap \Gamma_{\frac{1}{K}}$.  The aperture of $\Lambda_\tau$ and the aperture of $\tau$ are approximately the same.

For any surface sector $\Lambda \subset \Gamma_{\frac{1}{K}}$ there is a Lorentz transformation $\mathcal{L}$ which maps $\Lambda$ diffeomorphically onto $\Gamma_{\frac{1}{K}}$.  (The precise definition of $\Gamma_{\frac{1}{K}}$ was arranged to make this work.)  The formula for ${\color{black}\mathcal{L}}$ is as follows.

{\color{black} Let} $\mathcal{L} : \Lambda(d(\Lambda), \eta) \rightarrow \Gamma_{\frac{1}{K}}$ {\color{black}be} defined as (away from $\{z=0\}$):
\begin{equation}
\left\{
\begin{array}{rll}
{\color{black} \nu_3} & \mapsto & {\color{black} \nu_3},\\
\frac{{\color{black} \nu_2}}{{\color{black} \nu_3}} & \mapsto & \frac{1}{d(\Lambda)}(\frac{{\color{black} \nu_2}}{{\color{black} \nu_3}} - \eta(\Lambda)),\\
\frac{{\color{black} \nu_1}}{{\color{black} \nu_3}} & \mapsto & \frac{1}{d(\Lambda)^2}(\frac{{\color{black} \nu_1}}{{\color{black} \nu_3}} - \eta(\Lambda)\cdot \frac{{\color{black} \nu_2}}{{\color{black} \nu_3}} + \frac{\eta(\Lambda)^2}{2}).
\end{array} \right.
\end{equation}

We can see that $\mathcal{L}$ is actually a linear transformation:
\begin{equation} \label{lortrans}
\left\{
\begin{array}{rll}
{\color{black} \nu_3} & \mapsto & {\color{black} \nu_3},\\
{\color{black} \nu_2} & \mapsto & \frac{1}{d(\Lambda)}({\color{black} \nu_2} - \eta(\Lambda){\color{black} \nu_3}),\\
{\color{black} \nu_1} & \mapsto & \frac{1}{d(\Lambda)^2}({\color{black} \nu_1} - \eta(\Lambda){\color{black} \nu_2} + \frac{\eta(\Lambda)^2}{2}{\color{black} \nu_3}).
\end{array} \right.
\end{equation}

This linear transformation $\mathcal{L}$ is called {\color{black}a} Lorentz rescaling.

Suppose that $\tau$ is a sector with $d(\tau) = s$, and let $\Lambda = \Lambda_\tau$. {\color{black} We then} study the rescaling map $\mathcal{L}$ defined in (\ref{lortrans}).  We will need to keep track of how this change of variables affects the characters in our inequalities, like sectors $\tau' \subset \tau$ and the regions $U_{\tau,R}$.

First, if $\Lambda' \subset \Lambda$ is a smaller surface sector, then $\mathcal{L}( \Lambda')$ is a surface sector of aperture $\sim s^{-1} d(\tau')$.





More precisely, since $\Lambda' \subseteq \Lambda$, we have
\begin{equation}\label{containmentbetweeentautildeandtau}
[{\color{black}\eta(\Lambda')}- d(\Lambda'), \eta(\Lambda') + d(\Lambda')] \subseteq [\eta(\Lambda)- d(\Lambda), \eta(\Lambda) + d(\Lambda)].
\end{equation}

\noindent By the above definition of $\mathcal{L}$, we can see that $\mathcal{L}(\Lambda')$ is defined as
$$\{({\color{black} \nu_1}, {\color{black} \nu_2}, {\color{black} \nu_3}) \in \Gamma_{\frac{1}{K}}: \frac{{\color{black} \nu_2}}{{\color{black} \nu_3}} \in [\frac{1}{d(\Lambda)}(\eta(\Lambda') - \eta(\Lambda)) - \frac{d(\Lambda')}{d(\Lambda)}, \frac{1}{d(\Lambda)}(\eta(\Lambda') - \eta(\Lambda)) + \frac{d(\Lambda')}{d(\Lambda)}]  \}.$$

We see that (\ref{containmentbetweeentautildeandtau}) implies the above range of ${\color{black} \nu_2 / \nu_3}$ is in $[-1, 1]$, and that $\mathcal{L}(\Lambda')$ is a {\color{black} surface sector} of aperture $\frac{d(\Lambda')}{d(\Lambda)}$ lying inside the whole $\Gamma_{\frac{1}{K}} = \mathcal{L}(\Lambda)$.

Next we consider how $\mathcal{L}$ affects sectors $\tau' \subset \tau$.  Suppose that $\Lambda_{\tau'}$ is a surface sector associated to $\tau'$.  Note that $\tau'$ is approximately the convex hull of $\Lambda_{\tau'}$.  Since taking convex hulls commutes with linear transformations, we see that $\mathcal{L}(\tau')$ is approximately the convex hull of $\mathcal{L}(\Lambda_{\tau'})$, which is a sector of aperture $\sim s^{-1} d(\tau')$.

Next we consider $\mathcal{L}(N_{R^{-1}}({\color{black}\Lambda}))$ for some $R > s^{-2}$.  Note that $N_{s^2}(\Lambda)$ is approximately $\tau(\Lambda)$, but if $R > s^{-2}$ then $N_{s^2}(\Lambda)$ is far from being a convex set.  The $R^{-1}$-neighborhood of $\Gamma_{\frac{1}{K}}$ is covered by sectors $\theta \subset \tau$ with $d(\theta) = R^{-1/2}$.  Therefore, $\mathcal{L}(N_{R^{-1}}({\color{black}\Lambda}))$ is covered by sectors $\mathcal{L}(\theta)$ with aperture $\sim s^{-1} R^{-1/2}$.  The union of these sectors is the $s^{-2} R^{-1}$-neighorhood of $\Gamma_{\frac{1}{K}}$.  In summary
$\mathcal{L}(N_{R^{-1}}({\color{black}\Lambda}))$ is approximately $N_{s^{-2} R^{-1}}(\Gamma_{\frac{1}{K}})$.

Next we consider how the adjoint transformation, $\mathcal{L}^*$, behaves on physical space.  It is standard that the adjoint transformation behaves naturally with respect to taking duals, so, if $\theta$ is a sector, then we have $\mathcal{L}(\theta)^* = \mathcal{L}^* (\theta^*)$.

Finally we consider how $\mathcal{L}^*$ affects the sets $U_{\tau,R}$.
Recall from (\ref{defUtau}) that if $\tau = \tau(s, \xi)$, then
\begin{equation} \label{defUtau'} U_{\tau, R} = \{ x \in \RR^3: |\mathbf{c}(\xi) \cdot x| \le R s^{2} \textrm{ and } | \mathbf{n}(\xi) \cdot x| \le R \textrm{ and } | \mathbf{t}(\xi) \cdot {\color{black}x} | \le R s \}. \end{equation}

\noindent There is an equivalent more conceptual description, which is useful for understanding $\mathcal{L}^* (U_{\tau, R})$.
\begin{equation} \label{defUtau'} U_{\tau, R} \approx \textrm{Convex Hull } ( \cup_{\theta \subset \tau, d(\theta) = R^{-1/2}} \theta^*  ). \end{equation}

Now let $\tau$ again denote a fixed sector with $d(\tau)=s$ and let $\mathcal{L}$ be the Lorentz rescaling that takes $\Lambda_\tau$ to $\Gamma_{\frac{1}{K}}$.

\begin{lemma} \label{LUtau} For any sector $\tau' \subset \tau$ and any $R \ge s^{-2}$,

$$ \mathcal{L}^* (U_{\tau', R}) = U_{\mathcal{L}(\tau'), s^2 R}. $$
	
\end{lemma}

\begin{proof}

\begin{align*}
 \mathcal{L}^* (U_{\tau', R}) \approx& \mathrm{Convex Hull } ( \cup_{\theta \subset \tau', d(\theta) = R^{-1/2}} \mathcal{L}^* \theta^*  )\\
\approx  & \textrm{Convex Hull } ( \cup_{\theta \subset \tau', d(\theta) = R^{-1/2}} \mathcal{L}(\theta)^*  )\\
\approx & \textrm{Convex Hull } ( \cup_{\theta \subset \mathcal{L}(\tau'), d(\theta) = s^{-1} R^{-1/2}} \theta^* ) \approx U_{\mathcal{L}(\tau'), s^2 R}. \qedhere
\end{align*}
\end{proof}

We have now gathered enough background about Lorentz rescaling to carry out our Lorentz rescaling arguments in the next {\color{black}two sections}.

\section{The Proof of Lemma~\ref{small ball}} \label{secparab}
In this section, we prove  Lemma~\ref{small ball}.  First we prove several lemmas about the ``locally constant property'' of $f_{\theta}$.

\begin{lemma}\label{lem: locally constant}
	Let $\theta\subset \mathbb{R}^n$ be a compact convex set which is symmetric about a center point $c(\theta)$.
	If $\text{supp} \hat{f}_{\theta}\subset \theta$ and $T_{\theta}=\theta^*=\{x: |x\cdot (y-c(\theta))| \leq 1 \text{~for ~ all~} y\in \theta\}$, then there exists {\color{black}a positive} function $\eta_{T_{\theta}}$  satisfying:
	\begin{enumerate}
		\item $\eta_{T_{\theta}}$ is essentially supported on $10T_{\theta}$ and rapidly decays away from it:  for any integer $N\geq 0$, there exists a constant $C_N$ such that $\eta_{T_{\theta}}(x)\leq C_N  (n(x, 10T_{\theta}))^{-N}$ where $n(x, 10T_{\theta})$ is the smallest positive integer $n$ such that $x \in n\cdot 10T_{\theta}$,
		\item  $\|\eta_{T_{\theta}}\|_{L^1}\lesssim 1$,
		\item
		\begin{equation}\label{locally constant}
		|f_{\theta}|\leq \sum_{T\pp T_{\theta}} c_T \chi_T\leq |f_{\theta}|\ast \eta_{T_{\theta}}
		\end{equation}
		{\color{black}where $c_T$ is defined as $\max_{x\in T} |f_{\theta}|(x)$} and the sum $\sum_{T\pp T_{\theta}}$ is over a finitely overlapping cover $\{T\}$ of $\mathbb{R}^n$ with each $T\pp T_{\theta}$.
	\end{enumerate}
\end{lemma}
\begin{proof}
	We bound $|f_{\theta}|$ by
	\begin{equation}
	|f_{\theta}|\leq \sum_{T\pp T_{\theta}} c_{T}\chi_T{\color{black}.}
	\end{equation}

	Let $\phi_{\theta}$ be a smooth bump function supported on $2\theta$ and $\phi_{\theta}=1$ on $\theta$.
	Since $\text{supp} \hat{f}_{\theta}\subset \theta$, we have  $\hat{f}_{\theta}=\hat{f}_{\theta}\phi_{\theta}$ and $f_{\theta}= f_{\theta}\ast \phi_{\theta}^{\vee}$. Let $\eta_{T_{\theta}}(x) =\underset{t\in x+10T_{\theta}}{\max} |\phi_{\theta}^{\vee}|(t)$.  By non stationary phase,  $\phi_{\theta}^{\vee}$ is a function essentially supported on $T_{\theta} =\theta^*$, $|\phi_{\theta}^{\vee} (x)|\leq C_N (n(x, T_{\theta}))^{-N}$ and $\|\phi_{\theta}^{\vee}\|_{L^1}\sim 1$,  so  $\eta_{T_{\theta}}$ satisfies (1) and (2).

	 For any $T\pp T_{\theta}$,
	\begin{align*}
	\max_{x\in T} |f_{\theta}|(x)&\leq \max_{x\in T}\int |f_{\theta}|(y) |\phi_{\theta}^{\vee}(x-y)|dy\\
	&\leq \min_{x\in T} \int|f_{\theta}|(y) \eta_{T_{\theta}}(x-y) dy
	\end{align*}
	because for each $y$, $\underset{x\in T}{\max} |\phi_{\theta}^{\vee}|(x-y) \leq \underset{x\in T}{\min} ~~ \underset{t\in x-y+10T_{\theta}}{\max}|\phi_{\theta}^{\vee}|(t)$.
\end{proof}

\begin{lemma}\label{lem: convolution}
	Let $\eta_{T_{\theta}}$ be defined as in Lemma~\ref{lem: locally constant} and $T\pp T_{\theta}$, then for any integer $N>0$,
	there exists a {\color{black}positive} function $w_{T} =1 $ on $10T$ and $w_{T}(x)\leq C_N(1+\text{dist}(x, T))^{-N}$ such that  for any $1\leq p<\infty$,	 \begin{equation}\label{convolution}
	 \int_T (|f_{\theta}|\ast \eta_{T_{\theta}} )^p \lesssim_p \int |f_{\theta}|^p w_T.
	 \end{equation}
\end{lemma}
\begin{proof}
	We only need to prove the lemma for $N$ sufficiently large (depending on $p$).
	
	The function $\eta_{T_{\theta}}$ satisfies
	\begin{equation}
	\eta_{T_{\theta}} \leq \sum_{T\pp T_{\theta}} C_T \chi_T
	\end{equation}
	where $C_{T} \cdot |T| \lesssim_N  n(T, T_{\theta})^{-N}$ for any large integer $N>0$ and $n(T, T_{\theta})$ is the smallest  $n\geq 1$ such that $ T\subset nT_{\theta}$.
	
	By H\"{o}lder's inequality,
	\begin{align*}
	\int_T (|f_{\theta}| \ast \eta_{T_{\theta}})^p &\leq \int_{T} (\sum_{T'\pp T_{\theta}} |f_{\theta}|\ast C_{T'}\chi_{T'}  )^p\\
	&{\color{black}=}\int_T (\sum_{T'\pp T_{\theta}} n(T', T_{\theta})^{-\frac{4(p-1)}{ p}} \cdot
	 n(T', T_{\theta})^{\frac{4(p-1)} {p}} |f_{\theta}|\ast C_{T'} \chi_{T'})^p
	 \\
	&\lesssim  (\sum_{T'\pp T_{\theta}} n(T', T_{\theta}) ^{-4})^{p-1}\cdot  \sum_{T'\pp T_{\theta}} n(T', T_{\theta})^{4(p-1)} \int_T(|f_{\theta}|\ast C_{T'}\chi_{T'})^p\\
	&\lesssim  \sum_{T'\pp T_{\theta}} n(T', T_{\theta})^{4(p-1)} \int_T(|f_{\theta}|\ast C_{T'} \chi_{T'})^p.
	\end{align*}
	
	Let $\chi_{T - T'}(x)$ be the characteristic function {\color{black}of} the Minkowski sum $T-T'= T+ (-T')$.  Then by Young's inequality,
	\begin{align*}
	\int_T (|f_{\theta}|\ast C_{T'} \chi_{T'})^p & \leq \int ((|f_{\theta}|\chi_{T-T'} )\ast (C_{T'}\chi_{T'}))^p \\
	&\lesssim_N n(T', T_{\theta})^{-pN}.  \int_{T-T'} |f_{\theta}|^p
	\end{align*}
	It suffices to choose $w_T(x)\sim_N \sum_{ \tilde{T}\pp T} n(\tilde{T}, T)^{-N} \chi_{\tilde{T}}(x)$.
\end{proof}
\begin{cor}\label{cor: convolution}
	If $U$ is tiled by $T\pp T_{\theta}$, then  for any $1\leq p<\infty$,
	\begin{equation}
	\int_U (|f_{\theta}|\ast \eta_{T_{\theta}} )^p \lesssim_p \int |f_{\theta}|^p w_U
	\end{equation}
	where $w_U{\color{black}\geq 0}$ is essentially supported on $10U$ and rapidly decays {\color{black}away from it}.
\end{cor}

Remark.  It is important that $w_U$ can be taken uniformly independent of the choice of $T$. To see this, simply notice that if $x \in nU$ and $x \notin (n-1)U$ then $x$ cannot be in $(n-1)T$ for any $T \subset U$. Moreover for any $m$, a point $x$ lies in  $mT$ for $\lesssim m^3$ different $T$ in a given tiling $\{T\}_{T\pp T_{\theta}}$ of $\mathbb{R}^3$.

\begin{lemma}\label{iterate}
	Let $\theta_1,\theta_2\subset \tau $ be two sectors of aperture $d(\theta_1)=d(\theta_2)={\color{black}K^{-1/2}}$, and $\text{dist}(\theta_1, \theta_2)\sim d(\tau) =s>{\color{black}K^{-1/2}}$, then for any {\color{black}functions $\mathrm{supp} \hat{f}_{\theta_1}\subset N_{\frac{1}{K}}\Gamma_{\frac{1}{K}} \cap \theta_1$ and $\mathrm{supp} \hat{f}_{\theta_2}\subset N_{\frac{1}{K}}\Gamma_{\frac{1}{K}} \cap \theta_2$},
	$$\sum_{{\color{black}B_{K^{1/2}}}\subset \mathbb{R}^3} 
	\int_{B_{K^{1/2}}} |f_{\theta_1}f_{\theta_2}|^2 \lesssim  {\color{black}s^{-1}} \sum_{{\color{black} B_K \subset \mathbb{R}^3}}  |{\color{black} B_K}|^{-1}
	\int |f_{\theta_1}|^2 {\color{black} w_{B_K}}  \int |f_{\theta_2}|^2 {\color{black} w_{B_K}}.$$
\end{lemma}
\begin{proof}
The proof is essentially a bilinear-Kakeya-style\footnote{Bilinear Kakeya is an elementary statement stating: Let $|\Bbb{T}_1|$ and $|\Bbb{T}_2|$ be two finite families of infinite strips in $\Bbb{R}^2$ such that each strip has width $1$. Assume further that each $T_1 \in \Bbb{T}_1$ and each $T_2 \in \Bbb{T}_2$ have their directions $\sim 1$-separated, then $\int_{\Bbb{R}^2} (\sum_{T_1 \in \Bbb{T}_1} \chi_{T_1})\cdot(\sum_{T_2 \in \Bbb{T}_2} \chi_{T_2}) \lesssim |\Bbb{T}_1|\cdot|\Bbb{T}_2|$.} estimate in $\mathbb{R}^2$ 
plus the  locally constant property  in Lemma~\ref{lem: locally constant}.  This proof is a simple case of the ball inflation theorem (Theorem 9.2 in \cite{BD2}) in the proof of the  Bourgain--Demeter decoupling theorem.
	Since {\color{black}$\mathrm{supp} \hat{f}_{\theta_j}\subset N_{\frac{1}{K}}\Gamma_{\frac{1}{K}} \cap \theta_j$ for $j = 1, 2$}, the Fourier support of $f_{\theta_j}$ lies inside    a box $\tilde{\theta}_j$ of dimensions {\color{black}$K^{-1/2} \times K^{-1} \times K^{-1}$} with a common {\color{black}$K^{-1}$}--side on the $\nu_3$--direction {\color{black}(Recall the $(\nu_1, \nu_2, \nu_3)$-coordinate system and the equation of $\Gamma_{\frac{1}{K}}$ from Section \ref{lorentz})}.  And $T_{\tilde{\theta}_j}=\tilde{\theta}_j^*$ becomes a {\color{black}slab} of dimensions {\color{black}$K^{1/2} \times K \times K$}.  Since $\text{dist}(\tilde{\theta_1}, \tilde{\theta_2}) = \text{dist}(\theta_1, \theta_2)=s$,
for  each $T_1\pp T_{\tilde{\theta}_1}, T_{2}\pp T_{\tilde{\theta}_2}$ and $T_1, T_2\subset {\color{black}B_K}$, {\color{black}we have} $|{\color{black}T_1 \cap T_2}| {\color{black} \sim} {\color{black}K^{1/2}\cdot (s^{-1}K^{1/2})\cdot K = s^{-1}K^2}$. {\color{black}Hence  the key inequality $|T_1 \cap T_2| \sim s^{-1} |B_K|^{-1} |T_1| |T_2|$ holds{\footnote{\color{black}Note: All arguments in this paper work if we dilate a convex body by a constant. If we replace $B_K$ by the slightly bigger $B_{10K}$, then it is possible for $T_1$ and $T_2$ to miss each other, hence we can only obtain ``$\lesssim$''instead of the above ``$\sim$''. However we only use ``$\lesssim$'' in the inequality below so ``$\lesssim$'' is good enough to have.}}.}
	
	 Using 	Lemma~\ref{lem: locally constant}, now we are ready to bound
	\begin{align*}
	\sum_{{\color{black}B_{K^{1/2}}}\subset {\color{black}B_K}} 
	\int_{B_{K^{1/2}}} |f_{\theta_1}f_{\theta_2}|^2	&\leq \sum_{\color{black}\substack{B_{K^{1/2}} \subset B_K \\ T_1\pp T_{\tilde{\theta}_1}, {\color{black}B_{K^{1/2}}}\cap {\color{black}T_1} \neq \emptyset \\ T_2\pp T_{\tilde{\theta}_2}, {\color{black}B_{K^{1/2}}}\cap T_2 \neq \emptyset}} |{\color{black}B_{K^{1/2}}}| c_{T_1}^2 c_{ T_2}^2\\
	&{\color{black}\lesssim s^{-1}}  |{\color{black}B_{K}}|^{-1} (\int_{\color{black}B_{K}} \sum_{T_1\pp T_{\tilde{\theta}_1}} c_{ T_1}^2 \chi_{T_1})(\int_{\color{black}B_{K}} \sum_{T_2\pp T_{\tilde{\theta}_2} }c_{T_2}^2 \chi_{T_2})\\
	&\leq {\color{black}s^{-1}}  |{\color{black}B_{K}}|^{-1} \int_{\color{black}B_{K}} ( |f_{\theta_1}|\ast \eta_{T_{\tilde{\theta}_1}} )^2  \int_{\color{black}B_{K}}  ( |f_{\theta_2}|\ast \eta_{T_{\tilde{\theta}_2}} )^2  \\
	\text{(Corollary~\ref{cor: convolution})}&\lesssim {\color{black}s^{-1}}  |{\color{black}B_{K}}|^{-1} \int |f_{\theta_1}|^2 w_{\color{black}B_{K}} \int |f_{\theta_2}|^2 w_{\color{black}B_{K}}. \qedhere
	\end{align*}
\end{proof}

\begin{lemma}\label{lem: parabola}
	Let $f$ be a function whose Fourier transform is supported on the $\frac{1}{K}$-neighborhood of $\Gamma_{\frac{1}{K}}$. For any $\delta > 0$,
	\begin{equation}\label{F4}
\|f\|_{L^4(\mathbb{R}^3)}^4 \leq C_{\delta} K^{\delta} \sum_{K^{-1/2}\leq s \leq 1}\sum_{d(\tau)=s}  \underset{U\pp U_{\tau,K}}{\sum} |U|^{-1} \|{\color{black}S_U f}\|_{L^2}^4.
	\end{equation}
\end{lemma}
\begin{proof}
	
	{\color{black}We induct on $K$. The base case $K \lesssim_{\delta} 1$ is easy by H\"{o}lder's inequality.}
	
	Let $1\ll K_0 \ll K^{\delta/{\color{black}10}}$. We tile $N_{\frac{1}{K}}(\Gamma_{\frac{1}{K}})$ with sectors $\tau$ of aperture $\frac{1}{K_0}$ and width $\frac{1}{K}$ and decompose $f=\sum_{d(\tau)=\frac{1}{K_0}} f_{\tau}$.
	
	Now $N_{\frac{1}{K}}(\Gamma_{\frac{1}{K}})$ is the $\frac{1}{K}$-neighborhood of an arc of a parabola of length 1, and each $\tau$ is the $\frac{1}{K}$-neighborhood of an arc of the parabola of length $\frac{1}{K_0}$.

The Bourgain--Guth argument \cite{BG} says the following.
At each point $f(x)=\sum_{\tau} f_{\tau}(x)$. Let $\tau^*$ {\color{black}satisfy} $\max_{\tau} |f_{\tau}|(x) = |f_{\tau^*}|(x)$. If $|f_{\tau^*}|(x)\geq 1/10 |f|(x)$, then $|f|^4(x)\lesssim \sum_{\tau} |f_{\tau}|^4(x)$. Otherwise, there exists {\color{black}a} $\tau^{**}$ such that $\text{dist} (\tau^{**},\tau^*)\geq 1/K_0$ and $|f_{\tau^*}|(x)\geq |f_{\tau^{**}}|(x)  \geq \frac{1}{2K_0} |f|(x)$.
Hence,
\begin{align*}
|f|^4 \lesssim \sum_{d(\tau)=1/K_0} |f_{\tau}|^4+ K_0^4 \sum_{\text{dist}(\tau_1, \tau_2)\geq 1/K_0} |f_{\tau_1}f_{\tau_2}|^2.
\end{align*}

For the {\color{black}integral of the} first term, we rescale $\tau$ to be the $K^{-1}K_0^2$-neighborhood of $\Gamma_{1/K}$ (the rescaling argument here is similar to the one in the proof of Lemma~\ref{general} in Section~\ref{secusingrescaling}, which we {\color{black}will} do with full details), then we apply {\color{black}the} induction {\color{black}hypothesis} on the scale $K/K_0^2<K$.

For the {\color{black}integral of the} second term, we decompose $f_{\tau_j}=\sum_{\theta_j\subset \tau_j, d(\theta_j)=K^{-1/2}} f_{\theta_j}$, $j=1,2$. The functions $f_{\theta_1}f_{\theta_2}$ are essentially  orthogonal because they have almost disjoint Fourier  support, as in the Fefferman--C{\'o}rdoba proof of restriction for the parabola {\color{black}\cite{F}\cite{C0}}.

Since $\text{dist}(\tau_1, \tau_2) \geq \frac{1}{K_0}$, the Minkowski sum  $(\theta_1+\theta_2) \cap (\theta_1' +\theta_2') = \emptyset$ for $\theta_j, \theta_j' \subset \tau_j$, $j=1,2$,  unless $\theta_1'\subset K_0 \theta_1$ and $\theta_2'\subset K_0 \theta_2$. Hence
\begin{align*}
\sum_{B_{K^{1/2}}\subset \mathbb{R}^3}\int_{B_{K^{1/2}}} |f_{\tau_1}f_{\tau_2}|^2 &\leq K_0^2 \sum_{B_{K^{1/2}}\subset \mathbb{R}^3} \sum_{\text{dist}(\theta_1, \theta_2)\geq 1/K_0} \int_{B_{K^{1/2}}} |f_{\theta_1}f_{\theta_2}|^2\\
	(\text{Lemma~\ref{iterate}}) &\lesssim K_0^3\sum_{B_K \subset \mathbb{R}^3} |B_K|^{-1}  \sum_{\text{dist}(\theta_1, \theta_2)\geq 1/K_0}\int |f_{\theta_1}|^2 w_{B_K} \int |f_{\theta_2}|^2 w_{B_K}\\
	&\lesssim K_0^3 \sum_{B_K\subset \mathbb{R}^3} |B_K|^{-1} \|S_{B_K}f\|_{L^2}^4. \qedhere
\end{align*}

The right-hand side of the final line corresponds to the $s=1$ term of the right-hand side of \eqref{F4}.
\end{proof}

We recall the statement of Lemma~\ref{small ball}.  Unwinding the definition of $S_K(r,K)$ it says the following:
\begin{prop}\label{thm: parabola}
	Let $f$ be a function whose Fourier transform is supported on the $\frac{1}{K}$-neighborhood of $\Gamma_{\frac{1}{K}}$. For any $\delta > 0$ and any $r\leq K$,
	\begin{equation}\label{parabola}
	\sum_{B_r\subset {\color{black}\mathbb{R}^3}} |B_r|^{-1} \|{\color{black}S_{B_r} f}\|_{L^2(B_r)}^4 \leq C_{\delta} K^{\delta} \sum_{K^{-1/2}\leq s \leq 1}\sum_{d(\tau)=s}  \underset{U\pp U_{\tau{\color{black}, K}}}{\sum} |U|^{-1} \|{\color{black}S_U f}\|_{L^2}^4.
	\end{equation}
\end{prop}

{\color{black}\begin{proof}
We take advantage that $\Gamma_{\frac{1}{K}}$ is well-approximated by a parabola at the scale $1/K$ and use an approach similar to Fefferman--C{\'o}rdoba's to bound the left-hand side of (\ref{parabola}) by (essentially) the left-hand side of (\ref{F4})\footnote{\color{black}Alternatively, one can blackbox the $L^4$ angular square function estimate by C{\'o}rdoba \cite{C} and have a slightly shorter proof. We present a self-contained proof here.}.

Since the smallest aperture in this proposition is $K^{-1/2}$, we use $\theta$ to denote a sector on $\Gamma_{\frac{1}{K}}$ of aperture $K^{-1/2}$ in the current proof.

Let $A_1, \ldots, A_{1000}$ be disjoint sets of $\theta$ such that each $\theta$ is in one of them and:


Within each $A_j$, if the Minkowski sum $(\theta_1 + \theta_2) \cap (\theta_1 '+ \theta_2 ') \neq \emptyset$, then $(\theta_1, \theta_2) = (\theta_1 ', \theta_2 ')$ or $(\theta_2 ', \theta_1 ').$ (*)

Similar to Fefferman--C{\'o}rdoba's proof, we show that if we take each $A_j$ to be a collection of sectors that are enough separated and on a short enough arc, then (*) holds. In fact, it suffices to justify (*) when the constraint $(\theta_1 + \theta_2) \cap (\theta_1 '+ \theta_2 ') \neq \emptyset$ is replaced by the weaker one below: $\pi_3((\theta_1 + \theta_2)) \cap \pi_3 ((\theta_1 '+ \theta_2 ')) \neq \emptyset$. Here $\pi_3$ is the standard projection to the first two coordinates in the $(\nu_1, \nu_2, \nu_3)$-coordinate system. 
But the projection of $\Gamma_{\frac{1}{K}}$ onto the first two coordinates is contained in the $\frac{2}{K}-$neighborhood of the parabola $\nu_2^2 = 2\nu_1$, and the projection of each $\theta$ is the corresponding cap inside that neighborhood. We use $Error$ to denote a number (the ``error term'') whose absolute value is $\leq 4K^{-1}$. If $x_1 + x_2 = a + Error$ and $x_1^2 + x_2^2 = b + Error$ with $a, b \leq 2$, then $(x_1 - x_2)^2 = 2b-a^2 + 7Error$. Hence $|x_1 - x_2| = \sqrt{|2b-a^2|} + 3\sqrt{Error}$. This would imply that the pair $(x_1, x_2)$ is determined by the pair $(a, b)$, up to a swap in order and up to changing within $100$ adjacent caps $\theta$.

We use $\tau$ to denote caps with aperture $r^{-1/2} \geq K^{-1/2}$ in the current proof. Consider the decomposition $f_j = \sum_{\theta \in A_j} f_{\theta}$ and let $f_{j, \tau} = \sum_{\theta \subset \tau, \theta \in A_j} f_{\theta}$.

By the property (*) and Plancherel, we have for a fixed $j$,
\begin{align}\label{tauFC}
\int_{\mathbb{R}^3} |f_j|^4 & = \int_{\mathbb{R}^3} |\sum_{\tau} f_{j, \tau}|^4 \nonumber\\
& = \int_{\mathbb{R}^3} \sum_{\tau_1, \tau_2, \tau_3, \tau_4: (\text{supp} f_{j, \tau_1} + \text{supp} f_{j, \tau_2}) \cap (\text{supp} f_{j, \tau_3} + \text{supp} f_{j, \tau_4}) \neq \emptyset} f_{j, \tau_1} f_{j, \tau_2} \bar{f}_{j, \tau_3} \bar{f}_{j, \tau_4} \nonumber\\
& = \int_{\mathbb{R}^3} \sum_{\tau_1, \tau_2} n_{\tau_1, \tau_2}|f_{j, \tau_1} f_{j, \tau_2}|^2 \nonumber\\
& \sim \int_{\mathbb{R}^3} (\sum_{\tau} |f_{j, \tau}|^2)^2
\end{align}
where $n_{\tau_1, \tau_2} = 1$ if $\tau_1 = \tau_2$ and $n_{\tau_1, \tau_2} = 4/2 = 2$ if $\tau_1 \neq \tau_2$.

By (\ref{tauFC}) we have
\begin{align}
\sum_{B_r\subset \mathbb{R}^3} |B_r|^{-1} \|S_{B_r} f\|_{L^2(B_r)}^4
& \lesssim \sum_{j=1}^{1000}\sum_{B_r\subset \mathbb{R}^3} |B_r|^{-1} \|S_{B_r} f_j\|_{L^2(B_r)}^4\nonumber\\
& \leq \sum_{j=1}^{1000} \sum_{B_r\subset \mathbb{R}^3} \|S_{B_r} f_j\|_{L^4(B_r)}^4\nonumber\\
& = \sum_{j=1}^{1000} \int_{\mathbb{R}^3} (\sum_{\tau} |f_{j,\tau}|^2)^2\nonumber\\
& \sim \sum_{j=1}^{1000} \int_{\mathbb{R}^3} |f_j|^4\nonumber\\
(\text{Lemma ~\ref{lem: parabola}}) & \leq C_{\delta}K^{\delta}  \sum_{j=1}^{1000} \sum_{K^{-1/2}\leq s \leq 1}\sum_{d(\tau)=s}  \underset{U\pp U_{\tau{\color{black}, K}}}{\sum} |U|^{-1} \|{\color{black}S_U f_j}\|_{L^2}^4\nonumber\\
& \lesssim C_{\delta}K^{\delta} \sum_{K^{-1/2}\leq s \leq 1}\sum_{d(\tau)=s}  \underset{U\pp U_{\tau{\color{black}, K}}}{\sum} |U|^{-1} \|{\color{black}S_U f}\|_{L^2}^4. \qedhere
\end{align}
\end{proof}}

\section{The proof of Lemma~\ref{general}} \label{secusingrescaling}
Now we prove Lemma~\ref{general} using the Lorentz rescaling.  First we recall the statement.

\begin{lemma*}
	For any $r_1< r_2 \leq r_3$,
	$$S_K(r_1, r_3)\leq \log r_2 \cdot S_K(r_1, r_2)\max_{r_2^{-1/2}\leq s \leq 1}S_K(s^2r_2,  s^2 r_3).$$
\end{lemma*}

\begin{proof} Suppose that $\hat f$ is supported on $N_{r_3^{-1}}(\Gamma_{\frac{1}{K}})$.  To bound $S_K(r_1, r_3)$, we need to bound
	
$$ \sum_{B_{r_1}\subset {\color{black} \mathbb{R}^3}} |B_{r_1}|^{-1} \|{\color{black} S_{B_{r_1}} f}\|_{L^2(B_{r_1})}^4 . 
$$

We can apply the definition of $S_K(r_1, r_2)$ {\color{black} and get} 

$$ \sum_{B_{r_1}\subset {\color{black} \mathbb{R}^3}} |B_{r_1}|^{-1} \|{\color{black} S_{B_{r_1}} f}\|_{L^2(B_{r_1})}^4 \le S_K(r_1, r_2) \sum_{r_2^{-1/2} \le s \le 1} \sum_{d(\tau) = s} \sum_{{\color{black}U_1 \pp U_{\tau, r_2}}} |U_1|^{-1} \|{\color{black} S_{U_1} f} \|_{L^2{\color{black}(U_1)}}^4  .$$

Recall that if $U\pp U_{ \tau,r}$, then $S_U f=( \underset{d(\theta')=r^{-1/2}, \theta'\subset \tau}{\sum} |f_{\theta'}|^2)^{\frac{1}{2}}|_{U}$. In particular, $S_{B_r} f= ( \underset{d(\theta')=r^{-1/2} }{\sum} |f_{\theta'}|^2)^{\frac{1}{2}}|_{B_r}$.



Using Lorentz rescaling, we will prove the following lemma:

\begin{lemma}\label{rescaling} For any sector $\tau$ with $d(\tau) = s$, 
	\begin{equation}\label{before}
	\sum_{{\color{black}U_1 \pp U_{\tau, r_2}}
	} |U_1|^{-1} \|{\color{black} S_{U_1}} f\|_{L^2(U_1)}^4 \leq S_K(s^2r_2, s^2r_3) \sum_{r_3^{-1/2}\leq s'\leq s} ~~\sum_{d(\tau')=s', \tau' \subset \tau} ~~\sum_{{\color{black} U\pp U_{\tau', r_3}}
	}  |U|^{-1} \|{\color{black}S_U f}\|_{L^2{\color{black}(U)}}^4.
	\end{equation}
\end{lemma}


We defer the proof of Lemma \ref{rescaling} {\color{black}to the end of this section}.  If we plug in Lemma \ref{rescaling} and expand everything, then we get Lemma \ref{general}:

$$  \sum_{B_{r_1}\subset {\color{black} \mathbb{R}^3}} |B_{r_1}|^{-1} \|{\color{black} S_{B_{r_1}} f}\|_{L^2(B_{r_1})}^4  \leq \log r_2 S_K(r_1, r_2) {\color{black} \max_{r_2^{-1/2}\leq s \leq 1}} S_K(s^2 r_2, s^2r_3) \sum_{r_3^{-1/2}\leq s'\leq 1}\sum_{d(\tau')=s'} \sum_{U\pp U_{\tau', r_3}} |U|^{-1}\|{\color{black}S_U f}\|_{L^2{\color{black}(U)}}^4. $$

\noindent The factor $\log r_2$ appears here for the following reason: after we expand, each sector $\tau'$ will appear at most $\log r_2$ times, because $\tau'$ lies in $\tau$ for at most $\log r_2$ sectors $\tau$ with $r_2^{-1/2} \le d(\tau) \le 1$.
\end{proof}

\begin{proof} [Proof of Lemma \ref{rescaling}]  The definition of $S_K(s^2 r_2, s^2 r_3)$ says that if $\hat h$ is supported on $N_{s^{-2} r_3^{-1}}(\Gamma_{\frac{1}{K}})$, then
	
\begin{equation} \label{Sonh}
\sum_{{\color{black}B_{s^2 r_2}}} |B_{s^2 r_2}|^{-1} \| {\color{black}S_{B_{s^2 r_2}} h} \|_{L^2(B_{s^2 r_2})}^4 \le S_K(s^2 r_2, s^2 r_3) \sum_{s^{-1} r_2^{-1/2} \le d(\tau'') \le 1} \sum_{{\color{black}U'' \pp U_{\tau'', s^2 r_3}}} |U''|^{-1} \| {\color{black}S_{U''} h} \|_{L^2{\color{black}(U'')}}^4.
\end{equation}


On the other hand, Lemma \ref{rescaling} says that if $\tau$ is a sector of $\Gamma_{\frac{1}{K}}$ with $d(\tau) = s$, and $\hat f_\tau$ is supported on $N_{r_3^{-1}}(\Gamma_{\frac{1}{K}}) \cap \tau$, then 

	\begin{equation}\label{SonfU2}
\sum_{{\color{black}U_1 \pp U_{\tau, r_2}}} |U_1|^{-1} \|{\color{black}S_{U_1} f} \|_{L^2(U_1)}^4 \leq S_K(s^2r_2, s^2r_3) \sum_{r_3^{-1/2}\leq s'\leq s} ~~\sum_{d(\tau')=s', \tau' \subset \tau} ~~\sum_{{\color{black}U\pp U_{\tau', r_3}}}  |U|^{-1} \|{\color{black}S_U f}\|_{L^2}^4.
\end{equation}

To connect them, we begin with a Lorentz transformation $\mathcal{L}$ so that $\mathcal{L}: \tau \cap \Gamma_{\frac{1}{K}} \rightarrow \Gamma_{\frac{1}{K}}$ is a diffeomorphism.  This $\mathcal{L}$ is constructed in Section \ref{lorentz}, where it is shown that $\mathcal{L}$ takes $N_{r_3^{-1}}(\Gamma_{\frac{1}{K}}) \cap \tau$ to $N_{s^{-2} r_3^{-1}}(\Gamma_{\frac{1}{K}})$.  Now we define $h$ by $\hat{h} = \hat{f}_{{\color{black}\tau}}(\mathcal{L}^{-1}(\cdot))${\color{black}. Moreover let ${\hat h_{\tau''}} = {\hat f_{\tau '}} (\mathcal{L}^{-1} (\cdot))$ where $\mathcal{L} (\tau') = \tau''$, see the item (1) below. We} see that $\hat h$ is supported on $N_{s^{-2} r_3^{-1}}(\Gamma_{\frac{1}{K}})$ and so $h$ obeys (\ref{Sonh}).  When we unwind the Lorentz transformations, we claim that (\ref{Sonh}) becomes (\ref{SonfU2}), which proves the lemma.  To see that this unwinding works as desired, we check how each piece transforms.

\begin{enumerate}
	\item If $\tau' \subset \tau$ is a sector of $\Gamma_{\frac{1}{K}}$ with aperture $d(\tau')$, then $\mathcal{L}(\tau')$ is a sector $\tau''$ of $\Gamma_{\frac{1}{K}}$ with $d(\tau'') = s^{-1} d(\tau')$, as we showed in Section \ref{lorentz}. {\color{black}In particular, $\mathcal{L}$ transforms a $\theta' \subset \tau$ with aperture $d(\theta')=r_3^{-1/2}$ into a sector with aperture $s^{-1} r_3^{-1/2}$, which appears in the definition of $S_{U''}h$.
	}
	
	\item $\mathcal{L}^*(U_{\tau', r_3}) = U_{\tau'', s^2 r_3}$.  Since $\tau'' = \mathcal{L}(\tau')$, this follows from Lemma \ref{LUtau}.
	
	\item $\mathcal{L}^* (U_{\tau, r_2}) = B_{s^2 r_2}$.  Note that $\mathcal{L}(\tau)$ is the sector corresponding to all of $\Gamma_{\frac{1}{K}}$, which is essentially the unit ball.  We will denote this sector just by $B_1$.  By Lemma \ref{LUtau}, $\mathcal{L}^* (U_{\tau, r_2}) = U_{B_1, s^2 r_2}$.  By definition, the right-hand side is the convex hull of the union of $\theta^*$ over all sectors $\theta$ of aperture $\sim s^{-1} r_2^{-\frac{1}{2}}$, and this is approximately the ball of radius $s^2 r_2$.
	
	
	\item The Jacobian factors from the change of variables work out the same on the left-hand side and the right-hand side.  Since both sides involve a volume to the power $-1$ times an $L^2$ norm to the power 4, the Jacobian factors are the same on both sides of the inequality. \qedhere
\end{enumerate}
\end{proof}

\end{document}